\documentclass[11pt,a4paper]{article}
\usepackage[utf8]{inputenc}
\usepackage[T1]{fontenc}

\usepackage{fourier} 

\usepackage{amsmath}
\usepackage{amsfonts}
\usepackage{amssymb}
\usepackage{amsthm}
\usepackage{ upgreek }

\usepackage{mathtools}

\usepackage{cleveref}
\usepackage{autonum} 
\usepackage{calc}

\usepackage{pgf,tikz}
\usetikzlibrary{automata,arrows,positioning,calc,patterns,fit,calc,
decorations.pathreplacing,matrix,tikzmark,shadows.blur}
 \usepgflibrary{fpu}

\usetikzlibrary{arrows,positioning,patterns}

\crefformat{equation}{(#2#1#3)}

\newtheorem{theorem}{Theorem}[section]
\newtheorem{proposition}[theorem]{Proposition}
\newtheorem{lemma}[theorem]{Lemma}
\newtheorem{corollary}[theorem]{Corollary}
\newtheorem{definition}[theorem]{Definition}

\theoremstyle{definition}

\newtheorem{remarkx}[theorem]{Remark}
\newenvironment{remark}
  {\pushQED{\qed}\remarkx}
  {\popQED\endremarkx}

\newtheorem{assumptionNamex}{Assumption}
\newenvironment{assumptionName}[1]
  {\renewcommand\theassumptionNamex{#1}
    \pushQED{\qed}\assumptionNamex}
  {\popQED\endassumptionNamex}

\numberwithin{equation}{section}

\input{racc.sty}

\let\emptyset\varnothing

\newcommand\mystrutsub{\vrule height 0pt depth 2.5pt width 0pt}   
\newcommand{\ptrace}[2]{{\mystrutsub}_{#2}{#1}}
\newcommand{\QSD}[1]{\mathring\pi^{B_{#1}}_0}
\newcommand{\pev}[1]{\mathring\lambda^{B_{#1}}_0}
\newcommand{\nextev}[1]{\mathring\lambda^{B_{#1}}_1}
\newcommand{\QSDsmash}[1]{\smash{\mathring\pi^{B_{#1}}_0}}

\newcommand{\specgap}[1]{\mathring\varrho_{#1}}
\newcommand{\optto}{\twoheadrightarrow}
\newcommand{\Hgap}{\widehat H_0}

\newcommand{\trunc}[1]{#1_{\rm{tr}}}

\newcommand{\Biggintpartplus}[1]{\Biggl\lceil#1\Biggr\rceil}

\newcommand{\X}{\cX} 
\newcommand{\MN}{\cM}

\DeclareMathOperator{\vspan}{span}       

\newcommand{\bra}[1]{\langle#1|}
\newcommand{\ket}[1]{|#1\rangle}
\newcommand{\braket}[2]{\langle#1|#2\rangle}
\newcommand{\ketbra}[2]{\ket{#1}\bra{#2}}
\newcommand{\braketK}[3]{\langle#1|#2|#3\rangle}
\newcommand{\eps}{\varepsilon}
\newcommand{\Hhat}{\widehat{H}}

\begin{document}


\title{Reducing metastable continuous-space Markov chains\\
to Markov chains on a finite set}
\author{Nils Berglund}
\date{}   

\maketitle

\begin{abstract}
\noindent
We consider continuous-space, discrete-time Markov chains on $\R^d$, that admit a finite number $N$ of metastable states. Our main motivation for investigating these processes is to analyse random Poincar\'e maps, which describe random perturbations of ordinary differential equations admitting several periodic orbits. We show that under a few general assumptions, which hold in many examples of interest, the kernels of these Markov chains admit $N$ eigenvalues exponentially close to $1$, which are separated from the remainder of the spectrum by a spectral gap that can be quantified. Our main result states that these Markov chains can be approximated, uniformly in time, by a finite Markov chain with $N$ states. The transition probabilities of the finite chain are exponentially close to first-passage probabilities at neighbourhoods of metastable states, when starting in suitable quasistationary distributions. 
\end{abstract}

\leftline{\small{\it Date.\/} March 22, 2023. Final version, accepted in 
Annales de l'Institut Henri Poincar\'e, January 15, 2024.}
\leftline{\small 2020 {\it Mathematical Subject Classification.\/} 
60J05,  
37H05   
(primary), 
60J35,  
34F05   
(secondary).
}


\section{Introduction and informal statement of results} 
\label{sec:intro} 

In this work, we are concerned with continuous-space, discrete-time Markov chains, depending on a small parameter $\sigma \geqs 0$, which reduce to a deterministic map when $\sigma = 0$. Our main motivation for considering these processes is related to the notion of random Poincar\'e maps. 
Consider a stochastic differential equation (SDE) in $\R^d$, which is a weak-noise perturbation of an ordinary differential equation (ODE), admitting a finite number of asymptotically stable periodic orbits. In the deterministic limiting case, it is useful to introduce a surface of section $\Sigma$, transverse to the flow, and to study the sequence of returns of an orbit to $\Sigma$. This allows in particular to study stability and bifurcations of periodic orbits in a systematic way. 

A similar notion can be introduced in the stochastic case, taking some care in defining what one means by returns to $\Sigma$: one has to require that sample paths make some excursion away from $\Sigma$ between returns, in order to avoid accumulation of intersection points. To our knowledge, this notion appeared first in the works~\cite{Weiss_Knobloch_1990} by Weiss and Knobloch, and~\cite{HitczenkoMedvedev} by Hitczenko and Medvedev. In~\cite{BerglundLandon}, random Poincar\'e maps were used to study the distribution of small oscillations in the stochastic FitzHugh--Nagumo equation. In~\cite{HitczenkoMedvedev1}, they allowed to characterise the effect of noise on elliptic bursting. Random Poincar\'e maps also proved useful in the analysis of mixed-mode oscillations in systems such as the stochastic Koper model, featuring a folded-node singularity~\cite{berglund2013random}, and in determining the distribution of transition points through an unstable periodic orbit~\cite{berglund2014noise}. 

The work~\cite{BaudelBerglund_2017} initiated a more systematic study of random Poincar\'e maps, from the point of view of spectral theory. Its main result is that under a metastable hierarchy assumption on the $N$ stable periodic orbits, the kernels describing random Poincar\'e maps have exactly $N$ eigenvalues that are exponentially close to $1$. In addition, the remaining part of the spectrum is separated from those $N$ leading eigenvalues by a spectral gap, scaling like the logarithm of the inverse of the noise intensity. Asymptotic expressions for the leading eigenvalues and eigenfunctions in terms of commitor functions were also obtained in~\cite{BaudelBerglund_2017}.

The present work concerns a more general class of continuous-space Markov chains, which contain random Poincar\'e maps as a particular case, but are not limited to them. For instance, they also include randomly perturbed deterministic maps, of the form 
\begin{equation}
 X_{n+1} = \Pi(X_n) + \sigma\xi_{n+1}\;,
\end{equation} 
where $\Pi$ is a deterministic map defined on a subset of $\R^d$, and the $\xi_n$ are independent, identically distributed random variables. Deterministic iterated maps are common in applications such as population dynamics and epidemiology, and it is natural to study their perturbation by weak noise.

\paragraph*{Main results.}

We now give an informal statement of the main assumptions and results of this work. A precise formal statement is given in Sections~\ref{sec:setup} and~\ref{sec:main_result} below. We consider Markov chains on $\X_0 \subset \R^d$, with kernel $K_\sigma$, where $\sigma$ measures the noise intensity. In the deterministic case $\sigma = 0$, we assume that 
\begin{equation}
 K_0(x, A) = \ind{\Pi(x)\in A}\;,
\end{equation} 
for a deterministic map $\Pi: \X_0\to\X_0$. Our main assumptions are the following.

\begin{enumerate}
\item   \textbf{Deterministic limit:}
The deterministic map $\Pi$ leaves a compact set $\X\subset\X_0$ invariant. It has $N$ asymptotically stable fixed points in $\X$, and all its other limit sets are unstable fixed points. The aim of this assumption is to ensure that the asymptotic dynamics spends most of the time near a finite set of fixed points. 

\item   \textbf{Large-deviation principle:}
For positive $\sigma$, the kernel $K_\sigma$ has a smooth density, and it obeys a large-deviation principle with good rate function $I$. The rate function $I$ will be used to define a notion of quasipotential, that describes the exponential asymptotics of transition times between metastable sets.

\item   \textbf{Recurrence:}
For $\sigma > 0$, the Markov chain is positive Harris recurrent. This means in particular that it will reach any open set in a time having finite expectation. In particular, when starting anywhere in the compact set $\X$, the expected return time to $\X$ is bounded by a finite quantity $E_\X(\sigma)$. This assumption is needed for the existence of a spectral gap, between the $N$ leading eigenvalues, and the remainder of the spectrum of $K_\sigma$. 

\item   \textbf{Positivity:}
The process satisfies a uniform positivity condition in the neighbourhood of the stable fixed points of $\Pi$. This is a more technical property, defined in Section~\ref{ssec:setup_positivity} below, which essentially amounts to a lower bound of Doeblin type on transition densities. This assumption guarantees that the process conditioned on remaining near a stable fixed point relaxes to a so-called quasistationary distribution. 
\end{enumerate}

While the first two assumptions are quite natural, it may seem more difficult to ensure the last two assumptions. However, we will show in Section~\ref{sec:applic} that they are actually satisfied under quite weak conditions for the processes we are interested in, namely random Poincar\'e maps and randomly perturbed iterated maps. In particular, we will show that $E_\X(\sigma)$ is at most of order $\log(\sigma^{-1})$ in these cases, while the large-deviation principle implies that it is always at least sub-exponential in $\sigma$, in the sense that $\e^{-\eta/\sigma^2} E_\X(\sigma) \to 0$ as $\sigma\to 0$ for any $\eta > 0$.  

Our first main result reads as follows.

\begin{proposition}[Proposition~\ref{prop:spectral_gap}]
For sufficiently small positive $\sigma$, the kernel $K_\sigma$ has exactly $N$ eigenvalues which are exponentially close to $1$. All remaining eigenvalues have a modulus smaller than 
$\varrho = \e^{-c/E_\X(\sigma)}$ for some constant $c>0$. 
\end{proposition}

Note that~\cite[Thm.~3.2]{BaudelBerglund_2017} provides sharper bounds on the $N$ leading eignvalues, under a more restrictive metastable hierarchy assumption (essentially, all transitions between fixed points should happen on different exponential timescales). In that case, those eigenvalues can also be shown to be real. Here, however, we do not make such an assumption.

The existence of a spectral gap already shows that after a time of order $1/E_\X(\sigma)$, the process will be close to a finite-dimensional subspace of the space of measures on $\X_0$. The difficulty is that it is not straightforward to connect this finite-dimensional space to quantities that have a probabilistic interpretation. Our second main result provides such a connection. 
To state it, we introduce the \emph{trace process} $(X_{\tau^{+,n}})_{n\geqs0}$ of the Markov chain. Here $\tau^{+,n}$ denotes the time of $n$th return of the chain to a suitably defined neighbourhood $\MN$ of the set of stable fixed points, given by a union of neighbourhoods $B_i$ of these points. 

\begin{theorem}[Theorem~\ref{thm:reduction}]
Let $m(\sigma)$ be a function satisfying 
\begin{equation}
 \lim_{\sigma\to0} \sigma^2 \log m(\sigma) = \theta 
\end{equation} 
for a sufficiently small parameter $\theta > 0$. Then the probability of $X_{\tau^{+,nm(\sigma)}}$ belonging to $B_j$, when starting in $B_i$, is well-approximated by the probability of a Markov chain $(Y_n)_{n\geqs0}$ with values in $\set{1,\dots,N}$ being in state $j$. The transition probabilities of this Markov chain are given, up to exponentially small multiplicative errors, by the probability of the trace process first hitting $B_j$ at time $m(\sigma)$, when starting in a quasistationary distribution on $B_i$.  
\end{theorem}

More precisely, we will show that  there exists a linear map $\cL$ between measures on $\MN$ and measures on $\set{1,\dots,N}$, such that $\fP Y_n^{-1} = \cL(\fP X_{\tau^{+,nm(\sigma)}}^{-1})$ for all $n\in\N_0$. 

We refer to the statement of Theorem~\ref{thm:reduction} below for a precise formulation of what we mean by being well-approximated. Essentially, the difference between the distributions of the two processes is bounded uniformly in time by an exponentially small quantity. The result is thus mostly useful on long timescales, when the process has had an opportunity to explore several metastable states. Then our result states that whenever the finite Markov chain $Y_n$ is in state $j$ with a probability that is not exponentially small, the process $X_{\tau^{+,nm(\sigma)}}$ will belong to $B_j$ with a probability that is exponentially close to it. 

\paragraph*{Related results.}

The problem of approximating Markov processes by Markov chains on a finite set has been investigated for a long time, in particular in the case of SDEs. The idea is already present in the monograph~\cite{FreidlinWentzell_book} by Freidlin and Wentzell, where Markov chain approximations are used for instance to investigate the exit problem, and to approximate invariant measures. There is however no quantitative statement on how well the Markov chain approximates the original process directly. 

The works~\cite{BEGK,BGK} by Bovier, Eckhoff, Gayrard and Klein investigate reversible diffusion processes, governed by gradient SDEs of the form 
\begin{equation}
\label{eq:SDE_rev} 
 \6x_t = -\nabla V(x_t)\6t + \sigma \6W_t\;,
\end{equation} 
where $V$ is a confining multiwell potential. These articles use a potential-theoretic approach, that was originally limited to the reversible case, but was extended by Landim, Mariani and Seo to more general diffusions~\cite{Landim_Mariani_Seo19}. One result in~\cite{BGK} is that under a metastable hierarchy assumption, the expectations of transition times between certain well-chosen metastable sets are close to expectations of similar transitions in a finite Markov chain. 

Because of their importance in simulation algorithms in molecular dynamics, in particular in kinetic Monte Carlo algorithms~\cite{Voter07}, these results prompted a series of works aiming at obtaining precise descriptions of the exit location of solutions of SDEs from metastable sets, see in particular the works by Di Ges\`u, Leli\`evre, Le Peutrec and Nectoux \cite{GLPN19,GLPN20a,LPN22}. These authors also emphasized the importance of quasistationary distributions (QSDs) in metastable states for the reduction problem~\cite{GLPN_2016}. 
See for instance the work~\cite{Champagnat_Villemonais_23} by Champagnat and Villemonais for a recent review on QSDs. In parallel, results on the spectrum of reversible diffusions of the form~\eqref{eq:SDE_rev} have been extended to non-reversible diffusions by Le Peutrec and Michel~\cite{LePeutrec_Michel_20}, using methods from semiclassical analysis. 

In a different direction, many works have investigated the metastable behaviour of conti\-nuous-time Markov chains on countable sets, arising either in statistical physics, or as spatial discretisation of SDEs. In~\cite{Beltran_Landim_2010}, Beltr\'an and Landim introduced in particular the idea of a trace process to obtain a reduced description of the dynamics, while in~\cite{Beltran_Landim_2015_martingale_approach} they introduced a martingale method to study the convergence of sequences of such processes with increasingly large state spaces. In~\cite{Landim_Loulakis_Mourragui_2018}, Landim, Loulakis and Mouragui obtained convergence of finite-dimensional distributions of the so-called order parameter to those of a finite Markov chain. See~\cite{Landim_review_2019} for an overview of these results, and~\cite{Landim_Seo_18} for related results on sequences of discretisations of SDEs. 

Finally, a recent approach based on solutions of Poisson equations managed to show convergence of time-rescaled solutions to  metastable SDEs to finite Markov chains, in the limit of the noise intensity going to zero. See the work~\cite{Rezakhanlou_Seo_2021scaling} by Rezakhanlou and Seo for the reversible case, and the work~\cite{Jungkyoung_Seo_22} by Lee and Seo for non-reversible cases with known invariant measure, of Gibbs type. An overview is found in~\cite{Seo2020}. 

The main difference between the present work and those mentioned above, apart from the fact that it concerns continuous-time Markov chains instead of SDEs or Markov chains on countable spaces, is that it splits the approximation question into two separate problems. The first one, which is the main focus of this work, is to show that there exists a finite Markov chain that provides a good approximation to the metastable process. The second one is to obtain sharp asymptotics on transition probabilities of the finite Markov chain. This question is addressed here only in the sense of logarithmic equivalence, which naturally follows from the large-deviation principles. Sharper asymptotics will hopefully be determined in the future. The present results show, however, that it is sufficient to obtain such sharper asymptotics when starting in suitable QSDs. 

We finally remark that there are many works analysing the dynamics of singularly perturbed Markov chains, such as $(Y_n)_{n\geqs0}$. See for instance 
\cite{Schweitzer_68,Hassin_Haviv_92,AvrachenkovLasserre99,YinZhang1,betz2016multi,Freidlin_Koralov_17}.

\paragraph*{Structure of the paper.}
Section~\ref{sec:setup} contains the detailed set-up of the Markov processes we are interested in, states the four main assumptions, and introduces useful objects such as the trace process and quasistationary distributions. Section~\ref{sec:main_result} contains the precise statements of the two main results mentioned above. In Section~\ref{sec:applic}, we show that most of the main assumptions do hold quite generally in the case of the two main applications we have in mind, namely randomly perturbed iterated maps and random Poincar\'e maps. Sections~\ref{sec:proof_spectral_gap} and~\ref{sec:proof_reduction} contain the proofs of the two main results. Finally, the appendix contains the proofs of some auxiliary results used in Sections~\ref{sec:setup} and~\ref{sec:applic}.  

\paragraph*{Acknowledgments.}
The author would like to thank Manon Baudel for useful discussions that inspired the early stages of this work. This work was supported by the ANR project PERISTOCH, ANR--19--CE40--0023.



\section{Set-up and assumptions} 
\label{sec:setup} 

Let $\X_0 \subset \R^d$ be an open, connected domain, and denote its Borel 
$\sigma$-algebra $\cB(\X_0)$ by $\cS_0$. Our object of study are families  
$\set{K_\sigma}_{0\leqs\sigma<\sigma_0}$ of Markov kernels on 
$(\X_0,\cS_0)$, such that $K_0$ is a singular kernel, called the 
deterministic limit, while for $\sigma>0$ the kernel $K_\sigma$ is 
positive Harris recurrent and admits a continuous density. 

We denote by $(X^\sigma_n)_{n\geqs0} = (X_n)_{n\geqs0}$ the Markov chain with 
kernel $K_\sigma$, starting from some specified initial distribution $\mu$, 
and write $\probin{\mu}{\cdot}$ and $\expecin{\mu}{\cdot}$ for the associated 
law and expectations. If $\mu=\delta_x$, we simply write $\probin{x}{\cdot}$ 
and $\expecin{x}{\cdot}$. Given $A\in\cS_0$ we will sometimes use the notation 
\begin{equation}
 \bigexpecin{A}{\cdot} = \sup_{x\in A}\bigexpecin{x}{\cdot}\;.
\end{equation}
For any set $A\in\cS_0$, we denote by 
\begin{equation}
 \tau_A(x) = \inf\bigsetsuch{n\geqs0}{X_n\in A}
 \qquad\text{and}\qquad 
 \tau^+_A(x) = \inf\bigsetsuch{n\geqs1}{X_n\in A}
\end{equation} 
the hitting time of $A$ and return time to $A$ of $(X_n)_{n\geqs0}$ starting 
in $x$ (with the convention that $\inf\emptyset=\infty$). Note that 
$\tau^+_A(x) = \tau_A(x)$ whenever $x\notin A$, while $0 = \tau_A(x) < 
\tau^+_A(x)$ when $x\in A$. We will drop the argument $x$ whenever it is clear 
from the context. 

The kernel $K_\sigma$ induces two Markov semigroups in the standard way:
for any bounded measurable test function $\varphi\in L^\infty$, we have 
\begin{equation}
\label{eq:K_Linf} 
 (K_\sigma\varphi)(x) = \int_{\X_0} K_\sigma(x,\6y)\varphi(y) 
 = \bigexpecin{x}{\varphi(X_1)}\;,
\end{equation}
while for any (signed) measure $\mu\in L^1$ we have 
\begin{equation}
\label{eq:K_L1} 
 (\mu K_\sigma)(\6y) = \int_{\X_0} \mu(\6x) K_\sigma(x,\6y) 
 = \bigprobin{\mu}{X_1\in\6y}\;.
\end{equation} 
For $n\in\N$, we denote by $K_\sigma^n$ the $n$-fold kernel, defined by 
$K_\sigma^1=K_\sigma$ and \begin{equation}
 K_\sigma^{n+1}(x,A) = \int_{\cX_0} K_\sigma^n(x,\6z) K_\sigma(z,A)
 = \bigprobin{x}{X_{n+1}\in A}\;. 
\end{equation}
We are going to need a number of more precise assumptions, which are detailed 
in the next subsections. These concern the deterministic limit kernel $K_0$, 
a large-deviation principle for $\sigma\to0$, as well as positive Harris 
recurrence and local uniform positivity assumptions guaranteeing convergence to 
a unique invariant distribution. 


\subsection{Singular deterministic limit} 
\label{ssec:setup_det}

Let $\Pi: \X_0 \to\X_0$ be a map of class $\cC^2$. Note that we do not 
assume that $\Pi$ is invertible. We would like $K_0$ to describe the 
deterministic dynamical system $X_{n+1} = \Pi(X_n)$, which amounts to setting 
\begin{equation}
 K_0(x,A) = \ind{\Pi(x)\in A}
 \qquad\forall A\in\cS_0\;.
\end{equation} 
The forward semigroup then takes the form of the composition operator 
\begin{equation}
 (K_0\varphi)(x) = \int_{\X_0} \ind{\Pi(x)\in \6y} \varphi(y) 
 = (\varphi\circ\Pi)(x)\;,
\end{equation} 
sometimes called \emph{Koopman operator} in the physics literature. 
The backward semigroup is given by the \emph{pushforward} operator 
\begin{equation}
 (\mu K_0)(A) = \int_{\X_0} \mu(\6x) \ind{\Pi(x)\in A} 
 = \mu(\Pi^{-1}(A))\;,
\end{equation} 
which is known as the \emph{transfer operator} or 
\emph{Ruelle--Perron--Frobenius operator} in dynamical systems theory. 

For $A\in\cS_0$, we write $\Pi^0(A)=A$, and define inductively, for any 
$n\geqs1$, $\Pi^n(A) = \Pi\circ\Pi^{n-1}(A)$ and $\Pi^{-n}(A) = 
\bigsetsuch{x\in\X_0}{\Pi^n(x)\in A}$ (note that the last set may be empty). The 
\emph{$\omega$-limit set} $\omega(x)$ of $x\in\X_0$ is the set of accumulation 
points of the \emph{forward orbit} $(\Pi^n(x))_{n\geqs0}$ as $n\to\infty$. The 
\emph{$\alpha$-limit set} $\alpha(x)$ of $x$ is defined as the set of 
accumulation points of the \emph{backward orbit} $(\Pi^{-n}(x))_{n\geqs0}$. 

A \emph{fixed point} $x^\star$ of $\Pi$ (that is, a point $x^\star\in\X_0$ 
satisfying $\Pi(x^\star)=x^\star$) is called \emph{linearly asymptotically 
stable} if the Jacobian matrix $\partial_x\Pi(x^\star)$ has a spectral radius 
strictly smaller than $1$, and \emph{linearly unstable} if it has a spectral 
radius strictly larger than $1$. 

\begin{assumptionName}{DET}[Deterministic limit]
\label{ass:DET} 
There exists a bounded, open connected set $\X\subset\X_0$ such that 
$\Pi(\X)\subset\X$. The map $\Pi$ admits finitely many limit 
sets in $\X$, which are either linearly asymptotically stable fixed points, 
denoted $x^\star_1, \dots, x^\star_N$, or linearly unstable fixed points. 
\end{assumptionName}

For each $j=1, \dots, N$, we let $B_j$ be a closed set, containing $x^\star_j$ 
in its interior, and such that $\Pi(B_j) \subset B_j$. We will assume that the 
diameter of all $B_j$ is bounded by a constant $\delta>0$, which we are going 
to take small, but which is independent of $\sigma$. We denote by
\begin{equation}
 \MN = \bigcup_{j=1}^N B_j
\end{equation} 
the \emph{metastable set} of the process. 

\begin{remark}
The case of $\Pi$ admitting finitely many periodic points $x^\star_i$ of bounded 
minimal period $m_i$ as $\omega$-limit sets (i.e., $\Pi^{m_i}(x^\star_i) = 
x^\star_i$ and $\Pi^n(x^\star_i) \neq x^\star_i$ for $1\leqs n\leqs m_i-1$) can 
be covered by considering the iterated kernel $K_\sigma^m$ instead of 
$K_\sigma$, where $m$ is the least common multiple of the periods $m_i$ of all 
periodic points. 
\end{remark}


\subsection{Large-deviation principle} 
\label{ssec:setup_lpd}

Our second assumption concerns the behaviour of the the kernel $K_\sigma$ for 
small positive $\sigma$.

\begin{assumptionName}{LDP}[Large-deviation principle]
\label{ass:LDP} 
The kernel $K_\sigma$ satisfies a large-deviation principle (LDP) with good rate 
function $I$. That is, there exists a lower semi-continuous function 
$I:\X_0\times\X_0\to\R_+$, with compact level sets, such that 
\begin{equation}
 \label{eq:ldp_O}
 \liminf_{\sigma\to 0} \sigma^2 
 \log K_\sigma(x,O) 
 \geqs -\inf_{y\in O} I(x,y)
\end{equation} 
holds for any open set $O\in\cS_0$ and any $x\in\X_0$, and 
\begin{equation}
 \label{eq:ldp_C}
 \limsup_{\sigma\to 0} \sigma^2 
 \log K_\sigma(x,C)
 \leqs -\inf_{y\in C} I(x,y)
\end{equation} 
holds for any closed set $C\in\cS_0$ and any $x\in\X_0$. 
Furthermore, $I(x,y) = 0$ if and only if $y = \Pi(x)$, and $I$ is continuous at 
$(x^\star,x^\star)$ whenever $\Pi(x^\star)=x^\star$. 
\end{assumptionName}

With any sequence $(x_0,x_1,\dots,x_n)$ of points in $\X_0$, we associate 
the rate function 
\begin{equation}
 I(x_0,x_1,\dots,x_n) = \sum_{j=1}^n I(x_{j-1},x_j)\;.
\end{equation} 
Then the probability of the Markov chain visiting small 
neighbourhoods of $x_0, \dots, x_n$ in this particular order is logarithmically 
equivalent to $\e^{-I(x_0,\dots,x_n)/\sigma^2}$. 
The \emph{quasipotential} between two points $x$ and $y$ is then defined as 
\begin{equation}
\label{eq:def_Vxy} 
 V(x,y) = \inf_{n\geqs1} \inf_{x_1,\dots,x_{n-1}\in\X_0} 
 I(x,x_1,\dots,x_{n-1},y)\;.
\end{equation}
It represents the cost of going from $x$ to $y$ in arbitrary time. 

For $1\leqs i\neq j\leqs N$, we denote by 
\begin{equation}
\label{eq:def_Hij} 
 H(i,j) = V(x^\star_i,x^\star_j)
\end{equation} 
the quasipotential between the stable fixed points $x^\star_i$ and 
$x^\star_j$, and we define 
\begin{equation}
\label{eq:def_H_0} 
 H_0 = \min_{i \neq j} H(i,j)\;.
\end{equation} 
An important role will be played by the so-called \emph{committor 
functions} $\smash{\probin{x}{\tau^+_{B_j} < \tau^+_{B_i}}}$ between different 
balls $B_i$ and $B_j$. The following standard consequence of the LDP shows that 
for starting points $x\in B_i$, $\probin{x}{\tau^+_{B_j} < 
\tau^+_{B_i}}$ behaves like $\smash{\e^{-H(i,j)/\sigma^2}}$.
We give its proof in Appendix~\ref{sec:proof_setup}.

\begin{proposition}[Large-deviation estimates on committor functions]
\label{prop:ldp_committor} 
For any $\eta>0$, there exists $\delta_0>0$ such that, if the diameter of the 
sets $B_i$ satisfies $\delta < \delta_0$, then 
\begin{align}
 \liminf_{\sigma\to0} \sigma^2\log \bigprobin{x}{\tau^+_{B_j} < \tau^+_{B_i}}
 &\geqs -H(i,j) - \eta \;, \\
 \limsup_{\sigma\to0} \sigma^2\log \bigprobin{x}{\tau^+_{B_j} < \tau^+_{B_i}}
 &\leqs -H(i,j) + \eta
 \label{eq:ldp_committor} 
\end{align} 
holds for all $x\in B_i$. 
\end{proposition}
%


\subsection{Positive Harris recurrence} 
\label{ssec:setup_rec}

Properties of irreducibility, recurrence and positive recurrence can be defined 
in terms of hitting and return times, as discussed in~\cite{Meyn_Tweedie_92}. 
Given a $\sigma$-finite reference measure $\mu$ such that $\mu(\X_0)>0$, 
the process $(X_n)_{n\geqs0}$ is \emph{$\mu$-irreducible} if 
$\probin{x}{\tau^+_A < \infty}>0$ whenever $\mu(A)>0$. It is \emph{Harris 
recurrent} if $\probin{x}{\tau^+_A < \infty} = 1$ whenever $\mu(A)>0$, which is 
equivalent to the process visiting $A$ infinitely often. In this case, it is 
known~\cite{Nummelin84} that the process admits an essentially unique invariant 
measure $\pi_0$, and for any $A\in\cS_0$ with $\pi_0(A)>0$ and any measurable 
$f\geqs0$, one has 
\begin{equation}
\label{eq:pi_ergodicity} 
 \pi_0(f) := \int_{\X_0} f(x)\pi_0(\6x) = \int_A \pi_0(\6x) 
\Biggexpecin{x}{\sum_{n=1}^{\tau^+_A} f(X_n)}\;.
\end{equation}
If $\pi_0$ can be normalized to a probability measure, the process is called 
\emph{positive Harris recurrent}. Setting $f=1$ in~\eqref{eq:pi_ergodicity} 
shows that this is the case if $\expecin{A}{\tau^+_A} < \infty$ for some $A$ 
with $0<\mu(A)<\infty$. An important role will be played by the quantity
\begin{equation}
\label{eq:EX} 
 E_\cX(\sigma) := \expecin{\cX}{\tau^+_\cX}
 = \sup_{x\in\cX} \expecin{x}{\tau^+_\cX}\;.
\end{equation} 
Here we will make the simplifying assumption, which is motivated by the 
applications we have in mind, that for $\sigma>0$, $K_\sigma(x,\cdot)$ is 
absolutely continuous with respect to Lebesgue measure. This will allow us to 
take Lebesgue measure as reference measure. By further assuming that the 
density $k_\sigma$ of $K_\sigma$ is continuous and strictly positive in 
$\X\times\X$, we guarantee that $K_\sigma(x,A)>0$ whenever 
$A\subset\X$ has positive Lebesgue measure, which amounts to an ellipticity 
condition. In addition, these $A$ are \emph{petite} sets in the sense 
of~\cite[Sect.~3]{Meyn_Tweedie_92}. 

\begin{assumptionName}{REC}[Density and positive Harris recurrence]
\label{ass:REC} 
Whenever $\sigma>0$, $K_\sigma$ admits a density $k_\sigma$ with 
respect to Lebesgue measure, that is, 
\begin{equation}
\label{eq:K_density} 
 K_\sigma(x,A) = \int_A k_\sigma(x,y)\6y
\end{equation} 
for any $x\in\X_0$ and any $A\in\cS_0$. The density $k^n_\sigma$ of 
$K^n_\sigma$ is continuous and strictly positive in $\X$ for all $n\in\N$. 
Furthermore, $K_\sigma$ is positive Harris recurrent (with respect to Lebesgue 
measure), and there exists $\sigma_0 > 0$ such that $E_\cX(\sigma) < \infty$ for 
all $\sigma\in(0,\sigma_0]$. 
\end{assumptionName}

It follows from~\cite[Thm.~4.6]{Meyn_Tweedie_92} that a sufficient condition 
for positive Harris recurrence is that there exist a \emph{Lyapunov 
function} $U:\X_0\to\R_+$, going to infinity as $x\to\infty$, and constants 
$\varepsilon > 0$, $a\geqs0$ satisfying the discrete drift condition 
\begin{equation}
\label{eq:Lyapunov} 
 (K_\sigma U)(x) \leqs U(x) - \varepsilon + a\ind{x\in\X}\;.
\end{equation} 
In addition, \cite[Thm.~4.3]{Meyn_Tweedie_92} shows that~\eqref{eq:Lyapunov} 
implies the bound 
\begin{equation}
 \expecin{x}{\tau^+_\X} \leqs \frac{1}{\varepsilon} U(x) 
 \qquad \forall x\in \X_0\;,
\end{equation} 
so that $E_\cX(\sigma)$ is indeed finite. Note that if $U$ is a Lyapunov function for $K_0$, then it is a good candidate for being a Lyapunov function for small positive $\sigma$. 

\begin{remark}
\label{rem:quasipotential} 
An alternative to assuming the existence of a Lyapunov function is to work with 
the process conditioned on staying in the set $\X$ forever, via Doob's $h$-transform 
(see for instance~\cite[App.~B]{BaudelBerglund_2017}). This has a negligible 
effect on spectral-theoretic results if we assume that there exists a constant 
$\theta > 0$ such that 
\begin{equation}
\min_{1\leqs i \leqs N} V(x^\star_i, y) 
\geqs \max_{1\leqs i\neq j\leqs N} H(i,j) + \theta
\end{equation} 
holds for all $y \in \X_0 \setminus \X$.
\end{remark}

The following result shows that the LDP also provides a 
rough estimate, of order $\e^{\eta/\sigma^2}$ with arbitrarily small 
$\eta>0$, for the mean hitting time of the metastable set $\MN$ when starting 
in $\X$. Its proof is postponed to Appendix~\ref{sec:proof_setup}.

\begin{proposition}[Mean hitting time of $\MN$]
\label{prop:ldp_EMN} 
For any $\eta>0$, there exist $\sigma_0, \delta_0>0$ such that one has 
$E_\cX(\sigma) \leqs \e^{\eta/\sigma^2}$, provided 
$0<\sigma<\sigma_0$ and the diameter of the $B_i$ is bounded by $\delta_0$.
\end{proposition}

We will however see that in many practical situations, it is 
possible to show that $E_\cX(\sigma)$ is much 
smaller, typically of order $\log(\sigma^{-1})$, which yields better spectral 
gap estimates.


\subsection{Trace process} 
\label{ssec:setup_trace}

A very important process is going to be the \emph{trace process} on a 
recurrent set $A\in\cS_0$ (i.e., such that $\bigprobin{x}{\tau^+_A < 
\infty} = 1$ for all $x\in A$). 

\begin{definition}[Trace process]
Let $A$ be a positive recurrent set. The \emph{trace process} on $A$ is defined 
as the Markov chain monitored only when staying in $A$. Its transition kernel 
is given by 
\begin{equation}
 \ptrace{K_\sigma}{A}(x,B)
 = \bigprobin{x}{X_{\tau^+_A}\in B}
\end{equation} 
for any $B\in\cS_0$. We denote this process by $\ptrace{(X_n)}{A}_{n\geqs0}$.
\end{definition}

Note that owing to the strong Markov property, $\ptrace{K_\sigma}{A}$ is a 
markovian kernel, meaning that $\ptrace{K_\sigma}{A}(x,A) = 1$ for all $x\in A$. 
Since $A$ is recurrent, it is also a stochastic kernel on 
$A$. It can be rewritten in the form 
\begin{equation}
 \ptrace{K_\sigma}{A}(x,B) = \sum_{n\geqs 1} 
 \bigprobin{x}{\tau^+_A = n, X_n\in B}\;,
\end{equation} 
and thus for $\sigma>0$ it admits the density 
\begin{equation}
\label{eq:density_trace} 
 \ptrace{k_\sigma}{A}(x,y) = \sum_{n\geqs 1} 
 \bigprobin{x}{\tau^+_A = n} k_\sigma^n(x,y) \ind{x\in A, y\in A}\;.
\end{equation} 
Assume from now on that $A$ is positive recurrent (i.e., 
$\bigexpecin{x}{\tau^+_A} < \infty$ for all $x\in A$). 
Applying~\eqref{eq:pi_ergodicity} to $f = \indicator{B}$ for $B\subset A$, we 
obtain 
\begin{equation}
\label{eq:pi_trace} 
 \pi_0(B) 
 = \int_A \pi_0(\6x) \Biggexpecin{x}{\sum_{n=1}^{\tau^+_A} \ind{X_n\in B}}
 = \int_A \pi_0(\6x) \bigprobin{x}{X_{\tau^+_A} \in B}\;,
\end{equation} 
showing that the restriction of $\pi_0$ to $A$ is invariant under the trace 
process. It follows that the measure $\ptrace{\pi_0}{A}$ defined by 
\begin{equation}
 \ptrace{\pi_0}{A}(B) = \frac{\pi_0(B)}{\pi_0(A)}
 \qquad 
 \forall B \in \cS_0 \colon B \subset A
\end{equation} 
is an invariant probability measure of the trace process on $A$. 

\begin{remark}[Transitivity of the trace]
One easily checks that the trace enjoys the following transitivity property: 
if $B\subset A$, then $\ptrace{(\ptrace{(X_n)}{A})}{B}_{n\geqs0} = \ptrace{(X_n)}{B}_{n\geqs0}$.
\end{remark}

It will be more convenient to work with kernels defined on a bounded set. 
This can be achieved by considering, instead of the original kernel $K_\sigma$, 
the kernel $\ptrace{K_\sigma}{\X}$ of the trace process on the bounded set 
$\cX$, which contains essentially the same dynamic information owing to 
\Cref{ass:REC}. 

To lighten the notation, we will from now on simply write $K$ instead of 
$\ptrace{K_\sigma}{\X}$, the parameter $\sigma>0$ being always fixed at a 
sufficiently small value. The density of $K$, denoted by $k$, is continuous 
and strictly positive in $\X$ by \Cref{ass:REC}. The Borel $\sigma$-algebra of 
$\X$ will be denoted $\cB(\X)=\cS$, and for any $A\in\cS$ we write $A^c$ 
instead of $\X\setminus A$.

Since $\X$ is bounded and $k$ is continuous, $K$ is a compact operator (that 
is, it maps every closed set in $\cS$ to a relatively compact set, i.e., a set 
with compact closure). The Riesz--Schauder 
theorem~\cite[Thm.~VI.15]{Reed_Simon_I} ensures that $K$ has discrete 
spectrum, with all eigenvalues except possibly $0$ having finite multiplicity. 
The eigenvalues are roots of the Fredholm determinant, introduced 
in~\cite{fredholm1903classe}. Jentzsch's extension of the Perron--Frobenius 
theorem~\cite{Jentzsch1912} states that the eigenvalue of largest module is 
real and positive, and that the associated eigenfunctions can be taken real and 
positive as well. 

We will denote by $(\lambda_i)_{i\in\N_0}$ the eigenvalues of $K$, ordered by 
decreasing modulus, and by $\pi_i$ and $\phi_i$ the left and right 
eigenfunctions, that is 
\begin{equation}
 (\pi_i K)(x) = \lambda_i \pi_i(x) 
 \qquad \text{and} \qquad 
 (K\phi_i)(x) = \lambda_i \phi_i(x)
\end{equation} 
for all $i\in\N_0$. We normalise the eigenfunctions in such a way that 
\begin{equation}
\label{eq:normalisation} 
 \pi_i(\phi_j) 
 := \int_\X \pi_i(x)\phi_j(x) \6x = \delta_{ij}\;,
\end{equation} 
which implies that the kernels with density $\phi_i(x)\pi_i(y)$ are projectors 
on invariant subspaces of $K$. In case the set of eigenfunctions is complete and 
all nonzero eigenvalues have equal algebraic and geometric multiplicity, we have 
the spectral decomposition 
\begin{equation}
 k^n(x,y) = \sum_{i\geqs0} \lambda_i^n \phi_i(x)\pi_i(y) 
 \qquad
 \forall n\in\N\;.
\end{equation} 
If some geometric multiplicities are smaller than the corresponding algebraic 
multiplicities, this decomposition will contain nontrivial Jordan blocks. 

Since $K$ is stochastic ($K(x,\X)=1$ for all $x\in\X$), we have in particular 
$\lambda_0 = 1$, while $\pi_0$ is the density of the invariant distribution of 
the process, and $\phi_0$ is identically equal to $1$. In what follows, we will 
usually identify signed measures and their density.


\subsection{Killed process and QSDs} 
\label{ssec:setup_killed}

Given $A\in\cS$, we denote by $K_A$ the kernel of the process 
$(X^A_n)_{n\geqs0}$ killed upon leaving $A$. Its density has the 
expression
\begin{equation}
 k_A(x,y) = k(x,y)\ind{x\in A,y\in A}\;.
\end{equation} 
If $A^c$ has positive Lebesgue measure, this is a substochastic process, which 
can be turned into a stochastic process on $A\cup\bigset{\partial}$, where 
$\partial$ denotes a cemetery state. The killing time of the process is given 
for all $x\in A$ by $\tau_\partial(x) = \tau_{A^c}(x) = \tau^+_{A^c}(x)$. 

Fredholm theory also applies to $K_A$, and we denote its eigen-elements 
by $\lambda^A_i$, $\pi^A_i$ and $\phi^A_i$. A major difference in the 
substochastic case is that the \emph{principal eigenvalue} $\lambda^A_0$ is 
strictly smaller than $1$. 
The left eigenfunction $\pi^A_0$ is a \emph{quasiergodic distribution} (QED) 
of the process, meaning that it satisfies 
\begin{equation}
 \bigpcondin{\pi^A_0}{X^A_n \in B}{\tau_{A^c} > n} = \pi^A_0(B) 
 \qquad \forall B \in \cS, \;\forall n\in\N\;.
\end{equation} 
It can also be checked that the killing time, when starting in the QED, is 
geometrically distributed with success probability $(1-\lambda^A_0)$, that is, 
\begin{equation}
 \bigprobin{\pi^A_0}{\tau_{A^c} = n} 
 = (\lambda^A_0)^{n-1}(1-\lambda^A_0) 
 \qquad \forall n\in\N\;
 \qquad \text{and} \qquad 
 \bigexpecin{\pi^A_0}{\tau_{A^c}} = \frac{1}{1-\lambda^A_0}\;. 
\end{equation}
If the spectral-gap condition $\abs{\lambda^A_1} < \lambda^A_0$ is satisfied, 
then one also has 
\begin{equation}
\label{eq:QSD_convergence} 
 \lim_{n\to\infty} 
 \bigpcondin{x}{X^A_n \in B}{\tau_{A^c} > n} = \pi^A_0(B) 
\end{equation} 
for all $x\in A$ and all $B\in\cS$, meaning that $\pi^A_0$ is also a \emph{quasistationary distribution} (QSD). We refer 
to~\cite{Collet_Martinez_SanMartin_book, Bianchi_Gaudilliere_2016, 
Champagnat_Villemonais_16, GLPN_2016, Champagnat_Villemonais_23} for 
proofs and further details on QSDs.


\subsection{Uniform positivity} 
\label{ssec:setup_positivity}

The last assumption we need is a form of ergodicity condition, which is 
a particular case of the uniform positivity condition used 
in~\cite{Birkhoff1957}, and a variant of Doeblin's condition for Markov chains 
suitable for substochastic processes (see also~\cite{Hairer_Mattingly_11} for 
related results). 

\begin{definition}[Uniform positivity condition]
We say that a (sub)stochastic Markov kernel $K_A$ on $A$ with density $k_A$ 
satisfies a \emph{uniform positivity condition} with parameters $n\in\N$ and 
$L>1$ if 
\begin{equation}
\label{eq:UPC} 
 \sup_{x\in A} k^n_A(x,y)\leqs
 L \inf_{x\in A} k^n_A(x,y)
\end{equation} 
holds for all $y\in A$. 
\end{definition}

\begin{remark}
A more general uniform positivity condition one encounters in the literature 
is that $s(x) \nu(B) \leqs K^n(x,B) \leqs L s(x) \nu(B)$ for a positive 
function $s$ and a positive measure $\nu$. The form we use here corresponds to 
a constant $s$, which is sufficient for our purposes since we are going to 
apply it to sets $A$ on which $K^n(x,\cdot)$ is bounded below. 
\end{remark}

We will only need uniform positivity to hold for certain trace processes 
killed upon hitting some metastable sets. More precisely, given $1\leqs i \leqs 
N$, let $\ptrace{K_{\sigma,B_i}}{\MN}$ be the kernel of the trace 
process on $\MN$, killed when in hits $\MN\setminus B_i$ (which is 
equivalent to the trace process leaving $B_i$). 

\begin{assumptionName}{POS}[Uniform positivity]
\label{ass:POS} 
There exist a constant $L\in(1,2)$, independent of $\sigma$, and an integer 
$n_0(\sigma)$, such that for each $1\leqs i \leqs N$, the kernel 
$\ptrace{K_{\sigma,B_i}}{\MN}$ satisfies a uniform positivity condition on 
$B_i$ with parameters $n_0(\sigma)$ and $L$. Furthermore, for any $\eta > 0$, 
there exists $\sigma_0(\eta) > 0$ such that 
\begin{equation}
 n_0(\sigma) \leqs \e^{\eta/\sigma^2}
\end{equation} 
holds for all $\sigma \in(0,\sigma_0]$. 
\end{assumptionName}

At first glance, it might seem difficult to prove that such a condition holds. 
In practice, however, we will often be in the following situation. We have a bad 
upper bound on the oscillation of $x\mapsto\ptrace{k_{\sigma,B_i}}{\MN}(x,y)$ 
valid on the whole domain (typically, this bound has order $\e^{C/\sigma^2}$ for 
some $C>0$), but we also have a much smaller bound, uniform in $\sigma$, when 
$x$ is only allowed to vary on a small ball, typically of radius $\sigma^2$. The 
two bounds can then be combined into a much better one by using a 
coupling argument, see Proposition~\ref{prop:coupling} in Appendix~\ref{ssec:proof_pos_gaussian_map}.


\section{Main results} 
\label{sec:main_result}

We assume throughout this section, without further mention, that the kernel $K = \ptrace{K_\sigma}{\X}$ satisfies Assumptions~\ref{ass:DET}, \ref{ass:LDP}, \ref{ass:REC} and \ref{ass:POS}. Our first main result concerns the spectrum of $K$. We give its proof in Section~\ref{sec:proof_spectral_gap}. 

\begin{proposition}[Spectral gap estimate]
\label{prop:spectral_gap} 
For any $\eta > 0$, there exist $\sigma_0 > 0$ and $\delta_0 > 0$ such that, 
if $\sigma\in(0,\sigma_0]$ and the diameter of the sets $B_i$ is bounded by $\delta_0$, 
then the kernel $K$ has exactly $N$ eigenvalues outside the disc $\setsuch{\lambda\in\C}{\abs{\lambda} \leqs \varrho}$, where 
\begin{equation}
 \varrho = \exp\biggset{-\frac{\log2}{4 E_\X(\sigma)}}\;.
\end{equation} 
Furthermore, these $N$ eigenvalues all belong to the disc of radius $\e^{-[H_0 - \eta]/\sigma^2}$, centred in $1$. 
\end{proposition}

\begin{remark}[Sharper estimates on the $N$ first eigenvalues]
In~\cite{BaudelBerglund_2017}, we obtained sharper estimates on the $N$ largest eigenvalues, 
in terms of committor functions between the $B_i$, under a more restrictive condition on the $H(i,j)$. The condition requires that the $B_i$ can be ordered in such a way that 
\begin{equation}
 \min_{j < i} H(i,j) 
 \leqs \min_{k < i} \min_{j \leqs i, j\neq k} H(k, j) - \theta 
 \qquad 
 \forall i\in\set{1,\dots,N}
\end{equation} 
holds for some $\theta > 0$. Since we do not make this assumption here, it is necessary to give 
a new proof of Proposition~\ref{prop:spectral_gap} in the current situation. The proof uses however 
the same tools as in~\cite{BaudelBerglund_2017}. 
\end{remark}

One consequence of Proposition~\ref{prop:spectral_gap} is that the variables of the sequence $(X_n)_{n\geqs0}$ will be at distance decreasing like $\varrho^n$ from a sequence $((\trunc{X})_n)_{n\geqs0}$, generated by the truncated kernel $\trunc{K}$, obtained by projecting $K$ on the space associated with its $N$ largest eigenvalues. This truncated kernel is given, in the basis of eigenfunctions of $K$, by a matrix of size $N$. However, it is not immediatly clear how this approximate sequence relates to the sequence of visited $B_i$. The following approximation result clarifies that point.

\begin{theorem}[Approximation by a finite Markov chain]
\label{thm:reduction} 
There exist constants $C, \theta_0 > 0$ such that the following holds for all $\theta \in (0,\theta_0]$. Let $i\in\set{1,\dots N}$, and let $m = m(\sigma)$ satisfy 
\begin{equation}
 \lim_{\sigma\to 0} \sigma^2 \log(m(\sigma)) = \theta\;.
\end{equation} 
Then for any $\eta > 0$, there exist $\sigma_0 > 0$ and $\delta_0 > 0$, such that if $\sigma\in(0,\sigma_0]$ and the diameter of the sets $B_i$ is bounded by $\delta_0$, then for any $x \in B_i$, one has 
\begin{equation}
\label{eq:approx_MC} 
 \bigabs{\bigprobin{x}{X_{\tau^{+,nm}_{\MN}} \in B_j} 
 - \bigprobin{i}{Y_n = j}}
 \leqs C\bigpar{\e^{-[\Hhat_{\min}-\eta]/\sigma^2}
 + \varrho^{nm}}
\end{equation} 
for all $n\in\N$ and all $j\in\set{1,\dots N}$. Here $\Hhat_{\min}$ is a constant satisfying  $\Hhat_{\min} \geqs H_0 - (N-1)\theta$, and $(Y_n)_{n\geqs0}$ is 
a Markov chain with transition matrix $P$, whose matrix elements satisfy 
\begin{equation}
\label{eq:def_Pij} 
 P_{ij} = 
 \bigprobin{\QSD{i}}{X_{\tau^{+,m}_{\MN}}\in B_j}
 \bigbrak{1+\bigOrder{\e^{-[\theta -\eta]/\sigma^2}}}\;, 
\end{equation} 
where $\QSD{i}$ is the QSD of the trace process $\ptrace{(X_n)}{\MN}$ killed when leaving $B_i$.
Furthermore, these matrix elements satisfy
\begin{equation} 
\label{eq:ldp_Pij} 
\e^{-[H(i,j)-p(i,j)\theta +\eta]/\sigma^2}
 \leqs P_{ij}
 \leqs \e^{-[H(i,j)-p(i,j)\theta-\eta]/\sigma^2}\;,
\end{equation}
where the $p(i,j)$ are integers in $\set{0,\dots,N-1}$.  
\end{theorem}

We give the proof in Section~\ref{sec:proof_reduction}, which also contains more precise information on the constants $\theta_0$ and $\Hhat_{\min}$. 

Note that the result applies to time-diluted versions of the trace process on $\MN$, by a factor $m$ scaling like $\e^{\theta/\sigma^2}$. The reason for the time dilution is that it gives the process a chance to approach the local QSD once it has reached one of the $B_i$, before jumping to another $B_j$.

Relation~\eqref{eq:ldp_Pij} is a direct consequence of the large-deviation 
principle, where $p(i,j)$ represents the length of the optimal path from $i$ to $j$, see 
Lemma~\ref{lem:ldp_msigma} below. 
The main idea of the proof is to construct bases $(\mu_i)_{1\leqs i\leqs N}$ and $(\psi_j)_{1\leqs j\leqs N}$ of the $N$-dimensional vector spaces that are respectively left-invariant and 
right-invariant under the truncated kernel. The basis functions satisfy the orthonormality 
relation $\pscal{\mu_i}{\psi_j} = \delta_{ij}$. Most importantly, Theorem~\ref{thm:approx} below provides the relation
\begin{equation}
\label{eq:Pi_Ynj} 
 \bigprobin{i}{Y_n = j}
 = \int_{\MN} \int_{\MN} 
 \mu_i(x) \Bigprobin{x}{X_{\tau^{+,nm}_{\MN}} \in \6y} \psi_j(y) \6x 
 =: \Bigexpecin{\mu_i}{\psi_j \pth{X_{\tau^{+,nm}_{\MN}}}}\;,
\end{equation} 
where $\tau^{+,n}_{\MN}$ is the $n$th return time to the metastable set $\MN$. 
We emphasize that the measure $\mu_i$ in~\eqref{eq:Pi_Ynj} is in general a signed measure, because it is a linear combination of eigenfunctions of the truncated kernel. Therefore, the expectation in~\eqref{eq:Pi_Ynj} should be interpreted as a linear combination of expectations with 
respect to the positive and negative parts of $\mu_i$.
The functions $\psi_j$ can also have negative values, though they are close to indicator functions of $B_j$. See Proposition~\ref{prop:norms_mupsi} for details. However, the $P_{ij}$ are indeed positive for $\sigma$ small enough, and $P$ is a stochastic matrix, as proved in Lemma~\ref{lem:P_stoch}.

Relation~\eqref{eq:Pi_Ynj} shows in which way the original and reduced process are coupled. In fact, there exists a linear map $\cL$ from the space of measures on $\X$ to those on $\set{1,\dots,N}$ such that 
\begin{equation}
 \fP^i Y_n^{-1} = \cL\bigpar{\fP^{\mu_i} X_{\tau^{+,nm}_{\MN}}^{-1}}
 \qquad \forall n\in\N_0\;,
\end{equation} 
given by $\cL(\mu_j) = \delta_j$. The map $\cL$ is of course highly non-injective, since it maps an infinite-dimensional space to a space of dimension $N$. Its kernel is a complement of the space of measures spanned by $\mu_1, \dots, \mu_N$, given by the space of measures $\mu$ such that $\expecin{\mu}{\psi_j} = 0$ for $j=1, \dots, N$.

Relation~\eqref{eq:approx_MC} shows in which sense the sequence of visited balls $B_i$ is close to the Markov chain $(Y_n)_{n\geqs0}$. Note that the error term $\varrho^{nm}$ converges to $0$ as $n$ increases. It actually becomes negligible as soon as 
\begin{equation}
 n \geqs \Hhat_{\min} \frac{E_\X(\sigma)}{\sigma^2 m(\sigma)}\;,
\end{equation} 
which already happens for $n\geqs1$ if one applies Proposition~\ref{prop:ldp_EMN} with $\eta$ small enough. 

The important part of the error term in~\eqref{eq:approx_MC} is thus given by 
$C\e^{-[\Hhat_{\min}-\eta]/\sigma^2}$. The point is that this error is \emph{uniform} in time $n$. Thus at any given time $n$, we know that the trace process is likely to be in a ball $B_i$ whenever the probability $\bigprobin{i}{Y_n = j}$ is not exponentially small. This information becomes useful on time scales that are long compared to the typical time of transitions between metastable sets. 

The process $(X_{\tau^{+,nm}_{\MN}})_{n\geqs0}$ can thus be approximated, up to an exponentially small error that is uniform in time, by a Markov chain with transition probabilities $P_{ij}$. Note that the error in the expression~\eqref{eq:def_Pij} for these probabilities is multiplicative. Our analysis does not provide more explicit expressions for these transition probabilities than the large-deviation estimate in Proposition~\ref{prop:Px_LDP}, but it shows that is is sufficient to know the probabilities of hitting the different balls $B_j$ when starting in the QSD on each $B_i$. One may hope that future development of the theory will provide sharper estimates.  


\section{Applications} 
\label{sec:applic} 

In this section, we show that most of the main assumptions are automatically satisfied for the two main applications we have in mind, namely randomly perturbed iterated maps, and random Poincar\'e maps. 


\subsection{Iterated maps with additive noise} 
\label{ssec:mapnoise}

Let $\X_0 = \R^d$ and consider the Markov chain given by 
\begin{equation}
 X_{n+1} = \Pi(X_n) + \sigma \xi_{n+1}\;,
\end{equation} 
where $\Pi:\R^d\to\R^d$ satisfies \ref{ass:DET}, and the $\xi_n$ are i.i.d.\ 
random variables taking values in $\R^d$. A typical example would be that the 
$\xi_n$ are centred, normal random variables with positive definite covariance 
matrix $\Sigma$ (that is, we assume $c_-\norm{\zeta}^2 \leqs 
\pscal{\zeta}{\Sigma\zeta} c_+ \leqs \norm{\zeta}^2$ for all $\zeta\in\R^d$, 
where $c_+\geqs c_- > 0$). 

The transition kernel of the chain $(X_n)_{n\geqs0}$ is given by 
\begin{equation}
 K_\sigma(x,A) = \bigprob{\Pi(x)+\sigma\xi_1 \in A}
 = \bigprob{\sigma\xi_1 \in A - \Pi(x)}
 \qquad 
 \forall A\in\cS_0\;.
\end{equation} 
We now examine the four assumptions one by one.

\subparagraph{\Cref{ass:DET}.} 

The existence of a set $\cX$ invariant under 
the map $\Pi$ is a classical growth condition that holds true for many 
discrete-time dynamical systems. Let us assume for simplicity that $\cX$ can be 
taken as a ball $\cB(R_0) = \setsuch{x\in\R^d}{\norm{x}<R_0}$. For later use, 
we shall make the somewhat stronger assumption that $\Pi$ maps any ball 
$\cB(R)$ of radius $R\geqs R_0$ into a smaller ball, namely there exists 
$\varepsilon_0>0$ such that 
\begin{equation}
\label{eq:map_contraction} 
 \norm{\Pi(x)}^2 \leqs \norm{x}^2 - \varepsilon_0 
 \qquad \forall x \colon \norm{x} \geqs R_0\;.
\end{equation} 
Checking the conditions on limit sets is in general no easy task, as it 
requires a good understanding of fixed points and periodic orbits, their basins 
of attraction, and their stable and unstable manifolds. However they are known 
to hold for a number of systems. See for 
instance~\cite{CoutinhoFernandezLimaMeyroneinc} for a non-trivial dynamical 
systems arising from genetic regulatory networks. 

\subparagraph{\Cref{ass:LDP}.} 

Assume the random variable $\xi_1$ satisfies a large-deviation principle with 
good rate function $I_0$. Then it is immediate to see that $K_\sigma(x,\cdot)$ 
satisfies~\eqref{eq:ldp_O} and~\eqref{eq:ldp_C} with the rate function 
\begin{equation}
 I(x,y) = I_0(y - \Pi(x))\;.
\end{equation} 
We see that $I$ vanishes only if $y=\Pi(x)$ provided $I(x)>0$ for $x\neq 0$, 
and is continuous at fixed points whenever $I_0$ is continuous at $0$. 

In particular, if $\xi_1$ has a centred normal distribution with covariance 
matrix $\Sigma$, then 
\begin{equation}
\label{eq:I_normal} 
 I(x,y) = \frac12 \pscal{y-\Pi(x)}{\Sigma^{-1}(y-\Pi(x))}
\end{equation} 
satisfies all required properties. 

\subparagraph{\Cref{ass:REC}.}

Assume $\xi_1$ has a continuous density $p$. Then we see that $K_\sigma$ admits 
the density 
\begin{equation}
 k_\sigma(x,y) = \frac{1}{\sigma^d} p\Biggpar{\frac{y-\Pi(x)}{\sigma}}\;,
\end{equation} 
as required. Furthermore, taking $U(x) = \norm{x}^2$ as Lyapunov function, we 
obtain 
\begin{align}
(K_\sigma U)(x) 
&= \bigexpecin{x}{\norm{\Pi(x) + \sigma\xi_1}^2} \\
&= \norm{\Pi(x)}^2 + 2 \sigma\pscal{\Pi(x)}{\expec{\xi_1}} 
+ \sigma^2 \expec{\norm{\xi_1}^2}\;.
\end{align}
Thus if we assume that $\xi_1$ has zero mean and its components have bounded 
variance, it follows from~\eqref{eq:map_contraction} that the discrete drift 
condition~\eqref{eq:Lyapunov} is satisfied provided $\sigma^2 < 
\varepsilon_0\bigbrak{\expec{\norm{\xi_1}^2}}^{-1}$.  

These properties clearly hold in the case of Gaussian $\xi_i$, for which the 
density is
\begin{equation}
\label{eq:gaussian_density} 
 k_\sigma(x,y) = \frac{1}{\cN} \e^{-I(x,y)/\sigma^2}\;, 
 \qquad 
 \cN = (2\pi \sigma^2)^{d/2}(\det\Sigma)^{1/2}
\end{equation} 
with $I$ given by~\eqref{eq:I_normal}.

\subparagraph{\Cref{ass:POS}.} 

In the case where $\xi_1$ follows a normal law, the following result based on 
the coupling argument in Proposition~\ref{prop:coupling} shows that the positivity 
condition holds for sufficiently small diameter of the $B_i$. 

\begin{proposition}[Positivity for Gaussian noise]
\label{prop:pos_gaussian_map} 
Assume $\xi_1$ follows a centred, normal law with positive definite covariance 
matrix $\Sigma$. Then there exist $\delta_0,\sigma_0>0$ such that, if the $B_i$ 
have a diameter bounded by $\delta_0$ and $0<\sigma<\sigma_0$, then 
\Cref{ass:POS} is satisfied for $n_0(\sigma)$ of order $\log(\sigma^{-1})$.  
\end{proposition}
\begin{proof}
See Appendix~\ref{ssec:proof_pos_gaussian_map}. 
\end{proof}

Recall that Proposition~\ref{prop:ldp_EMN} shows that $E_\X(\sigma)$ is 
bounded by any exponential $\e^{\eta/\sigma^2}$ if $\sigma$ and the $B_i$ are 
small enough. In fact, we can do much better, and show that this expectation 
has order $\log(\sigma^{-1})$ if we assume that the deterministic system does 
not admit any heteroclinic cycles.
A \emph{heteroclinic orbit} from an unstable fixed point $z^\star_1$ to 
an unstable fixed point $z^\star_2$ is an orbit whose $\alpha$-limit set is 
equal to $z^\star_1$ and whose $\omega$-limit set is equal to $z^\star_2$. A 
\emph{heteroclinic cycle} between unstable fixed points $z^\star_1, \dots,  
z^\star_n$ is a set of heteroclinic orbits connecting $z^\star_1$ to 
$z^\star_2$, $z^\star_2$ to $z^\star_3$, \dots, $z^\star_{n-1}$ to $z^\star_n$ 
and $z^\star_n$ to $z^\star_1$.  

\begin{proposition}[Expected hitting time of $\MN$]
\label{prop:EMN_gaussian} 
Assume $\xi_1$ follows a centred, normal law with positive definite covariance 
matrix $\Sigma$, and the deterministic dynamical system generated by $\Pi_0$ 
has no heteroclinic cycles. Then there exist constants $c_0, \sigma_0, 
\delta_0>0$ such that 
\begin{equation}
 E_\X(\sigma) \leqs c_0 \log(\sigma^{-1})
\end{equation} 
holds for $0 < \sigma < \sigma_0$ and $0 < \delta < \delta_0$. 
\end{proposition}
\begin{proof}
See Appendix~\ref{ssec:proof_EMN_gaussian}. 
\end{proof}

The reason we exclude heteroclinic cycles is that the system may spend times 
longer than $\log(\sigma^{-1})$ in their neighbourhood. 
Note that SDEs with heteroclinic cycles have been investigated, 
for instance, in~\cite{Bakhtin_2011}. Results from that work may be transposed 
to the present situation, to analyse that point in more detail. 

Based on what is known in the continuous-time case~\cite{BaudelBerglund_2017}, 
similar results are expected to hold for more general systems with 
state-dependent noise, of the form 
\begin{equation}
 X_{n+1} = \Pi(X_n) + \sigma g(X_n)\xi_{n+1}\;, 
\end{equation} 
provided $g$ satisfies an ellipticity condition (that is, 
$g(x)\transpose{g(x)}$ should be positive definite). If $g$ fails to be 
elliptic at certain points, a more careful analysis becomes necessary.


\subsection{Random Poincar\'e maps} 
\label{ssec:poincare}

Consider a stochastic differential equation on $\cD_0\subset\R^{d+1}$ of the 
form 
\begin{equation}
\label{eq:SDE} 
 \6z_t = f(z_t) \6t + \sigma g(z_t) \6W_t\;,
\end{equation} 
where $f:\cD_0\to\R^{d+1}$ is a vector field of class $\cC^2$,  
$g:\cD_0\to\R^{(d+1)\times k}$ is of class $\cC^1$ and $(W_t)_{t\geqs0}$ is a 
$k$-dimensional standard Wiener process. Assume further that the deterministic 
ordinary differential equation 
\begin{equation}
\label{eq:ODE} 
 \dot z = f(z)
\end{equation} 
admits $N\geqs 2$ linearly asymptotically stable periodic orbits $\Gamma_1, 
\dots, \Gamma_N$, and that there exists a smooth $d$-dimensional manifold 
$\Sigma$ that all $\Gamma_i$ intersect transversally 
(cf.~\cite[Sect.~2.2]{BaudelBerglund_2017}). 

The \emph{random Poincar\'e map} associated with this system describes the 
sequence $(X_0, X_1, \dots)$ of successive intersections of a sample path 
$(z_t)_{t\geqs0}$ of the SDE~\eqref{eq:SDE} with $\Sigma$. To obtain a 
well-defined process, these intersections should be separated by excursions away 
from $\Sigma$, which can be achieved by requiring the sample path to visit 
another section $\Sigma'$, disjoint from $\Sigma$, between two consecutive $X_i$ 
(see~\cite[Sect.~2.3]{BaudelBerglund_2017}). The strong Markov property 
implies that the sequence $(X_n)_{n\geqs0}$ forms a Markov chain which, under 
suitable assumptions on $f$ and $g$, is of the form studied here.

\subparagraph{\Cref{ass:DET}.} This assumption is fulfilled if the 
deterministic system~\eqref{eq:ODE} admits a positively invariant, bounded open 
connected set $\cD\subset\cD_0$, intersecting $\Sigma$, and the limit sets 
of~\eqref{eq:ODE} are given by the $\Gamma_i$ and finitely many linearly 
unstable stationary points or unstable orbits. We can then take $\X = 
\cD\cap\Sigma$, and $\Pi$ maps a point in $x\in\Sigma$ to the point where the 
positive orbit of $x$ first returns to $\Sigma$. 
Furthermore, $x^\star_i = \Gamma_i\cap\Sigma$, and the intersections of the 
unstable periodic orbits with $\Sigma$ are the unstable fixed points of $\Pi$. 

\subparagraph{\Cref{ass:LDP}} Assume the diffusion coefficient $g$ satisfies 
an ellipticity condition, that is, there exist constants $c_+ \geqs c_- > 
0$ such that the diffusion matrix $D(z) = g(z)\transpose{g(z)}$ satisfies 
\begin{equation}
\label{eq:ellipticity} 
 c_- \norm{\xi}^2 \leqs \pscal{\xi}{D(z)\xi} 
 \leqs c_+ \norm{\xi^2}
\end{equation} 
for all $z\in\cD$ and $\xi\in\R^{d+1}$. Then Wentzell--Freidlin 
theory~\cite{FreidlinWentzell_book} provides the existence of a 
sample-path LDP with rate function 
\begin{equation}
\mathcal{I}_{[0,T]}(\gamma) = 
 \begin{cases}
 \displaystyle \frac12 
 \int_0^T \transpose{(\dot\gamma_s - f(\gamma_s))}D(\gamma_s)^{-1} 
(\dot\gamma_s - f(\gamma_s)) \6s & \text{if $\gamma\in H^1$\;,} \\
+ \infty & \text{otherwise\;.}
 \end{cases}
\end{equation} 
Note that if $g$ fails to satisfy the ellipticity 
condition~\eqref{eq:ellipticity}, an LDP may still hold, but its rate function 
is given by a variational principle (obtained by applying the contraction 
principle to Schilder's theorem for scaled Brownian motion). 

This continuous-time LDP induces, by the contraction 
principle, a discrete-time LDP with rate function 
\begin{equation}
 I(x,y) = \inf_{T>0} \; \inf_{\gamma\colon x\to y} 
\mathcal{I}_{[0,T]}(\gamma)\;,
\end{equation} 
where the second infimum runs over paths connecting points $x$ and $y$ 
in $\Sigma$ in time $T$, and making an excursion via $\Sigma'$. 

\subparagraph{\Cref{ass:REC}.} The ellipticity condition~\eqref{eq:ellipticity} 
ensures that the kernel $K_\sigma$ of the Markov chain $(X_n)_{n\geqs0}$ admits 
a continuous density $k_\sigma$ as shown 
in~\cite{BenArous_Kusuoka_Stroock_1984}. In fact, a weaker hypo-ellipticity 
condition is sufficient. As for Harris recurrence, it follows from a 
continuous-time analogue of the discrete drift 
condition~\eqref{eq:Lyapunov}~\cite{Meyn_Tweedie_1993b}.
Namely, there should exist a function $V:\cD_0\to\R_+$ of class $\cC^2$, 
diverging as $\norm{x}\to\infty$, and constants $c>0$ and $d\geqs 0$ such that 
\begin{equation}
 (\cL V)(z) \leqs -c + d \ind{z\in\cD}
 \qquad \forall z\in\cD_0\;,
\end{equation} 
where $\cL$ is the infinitesimal generator of the diffusion~\eqref{eq:SDE}. 

\subparagraph{\Cref{ass:POS}.} The uniform positivity condition~\eqref{eq:UPC} 
of the trace process can again be proved to hold by applying the coupling 
argument of Proposition~\ref{prop:coupling}. Instead of using Harnack inequalities for the 
density of a Gaussian random variable, one can use Harnack inequalities 
satisfied by harmonic functions, see~\cite[Sect.~5.1]{BaudelBerglund_2017} 
and~\cite[Sect.~5.3]{berglund2014noise}. The parameter $n_0(\sigma)$ in the 
uniform positivity condition has again order $\log(\sigma^{-1})$. 

We also have an analogue of Proposition~\ref{prop:EMN_gaussian} on the expected hitting 
time of the metastable set $\MN$. 

\begin{proposition}[Expected hitting time of $\MN$]
\label{prop:EMN_poincare} 
Assume the deterministic dynamical system \eqref{eq:ODE} has no heteroclinic 
cycles. Then there exist constants $c_0, \sigma_0, \delta_0>0$ such that 
\begin{equation}
 E_\X(\sigma) \leqs c_0 \log(\sigma^{-1})
\end{equation} 
holds for $0 < \sigma < \sigma_0$ and $0 < \delta < \delta_0$. 
\end{proposition}
\begin{proof}
See~\cite[Cor.~8.13]{BaudelBerglund_2017}. This work excluded the 
existence of heteroclinic orbits between unstable periodic orbits, but the same 
arguments as in the proof of Proposition~\ref{prop:EMN_gaussian} show that the absence of 
heteroclinic \emph{cycles} is sufficient.  
\end{proof}




\section{Proof of Proposition~\ref{prop:spectral_gap}}
\label{sec:proof_spectral_gap} 

In this section, we give the proof of the spectral gap result stated in 
Proposition~\ref{prop:spectral_gap}, by adapting the proof 
of~\cite[Thm~3.2]{BaudelBerglund_2017} to the weaker assumptions of the present work. 
We start with a simple but useful a priori estimate.

\begin{lemma}
\label{lem:exit_M} 
There exist constants $\theta_0, \sigma_0 > 0$ such that 
\begin{equation}
 \sup_{x\in\MN} \bigprobin{x}{X_1 \in \MN^c} \leqs \e^{-\theta_0/\sigma^2}
 \label{eq:bound_ldp_MNc} 
\end{equation} 
holds for all $\sigma \leqs \sigma_0$. 
\end{lemma}
\begin{proof}
Pick an $x\in B_i\subset\MN$. Since $B_i$ is assumed to be positively invariant under 
the map $\Pi$ and $x^\star_i$ is asymptotically stable, $\Pi(x)$ belongs to $B_i$, and 
its distance to $\partial B_i$ is bounded below. The claim thus follows from the large-deviation principle. 
\end{proof}

Fix $x\in\X$ and $m_0\in\N$. By Markov's inequality and the definition~\eqref{eq:EX} 
of $E_\X(\sigma)$, we have 
\begin{equation}
 \bigprobin{x}{\tau^+_\MN > m_0}
 \leqs \frac{1}{m_0} \expecin{x}{\tau^+_\MN} 
 \leqs \frac{1}{m_0} E_\X(\sigma)\;.
\end{equation} 
This shows that 
\begin{align}
 \bigprobin{x}{X_{m_0}\notin\MN}
 &\leqs \bigprobin{x}{\tau^+_\MN > m_0} 
 + \bigprobin{x}{X_{m_0}\notin\MN, \tau^+_\MN \leqs m_0} \\
 &\leqs \frac{1}{m_0} E_\X(\sigma) 
 + \sup_{m_1\leqs m_0} \sup_{y\in\MN} \bigprobin{y}{X_{m_1}\in\MN^c}\;.
\end{align}
The second term is exponentially small by~\eqref{eq:bound_ldp_MNc}, and is thus 
bounded by $\frac14$ for $\sigma$ small enough.  
For sufficiently small $\sigma$, we thus have 
\begin{equation}
 \bigprobin{x}{X_{m_0}\notin\MN} \leqs \frac12
\end{equation} 
provided $m_0 \geqs 4E_\X(\sigma)$. 


\subsection{Feynman--Kac representation formulas}
\label{ssec:FK} 

We will now rely on Feynman--Kac representation formulas for eigenfunctions of the 
kernel $K$, as used in~\cite[Sect.~4]{BaudelBerglund_2017}, to show that there 
are only $N$ eigenvalues outside a given disc in the complex plane.  
Writing $\widetilde X_n = X_{nm_0}$ for the time-diluted 
Markov chain and $\tilde\tau^+_\MN$ for the corresponding first-hitting time of $\MN$, 
we obtain from~\cite[Lem.~4.1]{BaudelBerglund_2017} that the Laplace transform 
$\bigexpecin{x}{\e^{u\tilde\tau^+_\MN}}$ exists whenever $\abs{\e^{-u}} \geqs \frac12$.
By~\cite[Cor.~4.3]{BaudelBerglund_2017}, we know that $(\e^{-u},\phi)$ is an 
eigenpair of $K^{m_0}$ for $\abs{\e^{-u}} > \frac12$ if, and only if, one has 
\begin{equation}
 ((K^u)^{m_0}\phi)(x) = \e^{-u} \phi(x) 
 \qquad \forall x \in \MN\;, 
\end{equation} 
where $K^u$ is the kernel defined by 
\begin{equation}
 K^u(x,A) = \biggexpecin{x}{\e^{u(\tau^+_\MN-1)}
 \ind{X_{\tau^+_\MN} \in A}} 
\qquad \forall A\in\cB(\MN)\;.
\end{equation} 
Note that $K^0$ is equal to the kernel $\ptrace{K}{\MN}$ of the trace process on $\MN$.
It follows that $(\e^{-u},\phi)$ is an 
eigenpair of $K$ for $\abs{\e^{-u}} > \bigpar{\frac12}^{1/m_0}$ if, and only if, one has 
\begin{equation}
\label{eq:Ku_eigenvalue} 
 (K^u\phi)(x) = \e^{-u} \phi(x) 
 \qquad \forall x \in \MN\;.  
\end{equation} 

\begin{remark}
It may seem unusual that the variable $u$ appears both in the kernel $K^u$ and the eigenvalue 
$\e^{-u}$. This is not a problem, however, since the eigenvalue problem~\eqref{eq:Ku_eigenvalue} can be considered as the system 
\begin{equation}
 K^u\phi = \lambda\phi\;, 
 \qquad 
 \lambda = \e^{-u}
\end{equation} 
for two unknowns $\lambda$ and $\e^{-u}$. 
\end{remark}

The idea is now to compare the kernel $K^u$ to the simpler kernel $K^\star$, 
defined for $A\subset\MN$ by 
\begin{equation}
 K^\star(x,A) 
 = \sum_{i=1}^N \ind{x\in B_i} \bigprobin{\QSD{i}}{X_{\tau^+_{\MN}}\in A}\;.  
\end{equation} 
Here we recall that $\QSD{i}$ denotes the quasistationary distribution of the trace process 
on $\MN$ killed when leaving $B_i$. Since 
\begin{equation}
 (K^\star\phi)(x) = \sum_{i=1}^N \ind{x\in B_i} 
 \bigexpecin{\QSD{i}}{\phi(X_{\tau^+_{\MN}})}\;,
\end{equation} 
the kernel $K^\star$ has finite rank. Indeed, its image is the $N$-dimensional space of 
functions $\phi:\MN\to\R$ that are constant on each $B_i$. Therefore, $K^\star$ has at most 
$N$ nonzero eigenvalues. These eigenvalues are exactly those of the $N$ by $N$ stochastic matrix $P^\star$ 
with elements 
\begin{equation}
\label{eq:Pij} 
 P^\star_{ij} = \bigprobin{\QSD{i}}{X_{\tau^+_\MN} \in B_j}\;.
\end{equation} 
Note that Proposition~\ref{prop:ldp_committor} implies that for any $\eta > 0$, there exists a 
$\sigma_0(\eta) > 0$ such that one has 
\begin{equation}
\label{eq:ldp_MN_Bi} 
 \e^{-(H_0 + \eta)/\sigma^2}
 \leqs \probin{x}{\tau_{\MN\setminus B_i} \leqs n} 
 \leqs n\e^{-(H_0 - \eta)/\sigma^2}
\end{equation} 
for any $n\in\N$, any $x\in B_i$ and all $\sigma < \sigma_0(\eta)$, where $H_0$ has been 
introduced in~\eqref{eq:def_H_0}. This shows in particular that the matrix elements~\eqref{eq:Pij} satisfy 
\begin{equation}
 P_{ij} \leqs  \e^{-(H_0 - \eta)/\sigma^2} 
 \qquad 
 \text{for $i\neq j$\;.}
\label{eq:bound_Pij} 
\end{equation} 


\subsection{Norm estimates on kernels}
\label{ssec:norm_kernel} 

In order to compare kernels, we will need a norm on the space of (signed) kernels on $\MN$. If $Q$ is 
such a kernel with density $q$, we write 
\begin{equation}
\label{eq:K_norm} 
 \norm{Q} = \sup_{x\in\MN} \int_{\MN} \bigabs{q(x,y)}\6y 
 = \sup_{x\in\MN} \abs{Q(x,\MN)}\;.
\end{equation} 
One easily checks that this is a subordinate norm, given by 
\begin{equation}
\label{eq:K_operator_norm} 
\norm{Q} 
= \sup_{\varphi\in L^\infty\colon\norm{\varphi}_\infty = 1} 
\norm{Q\varphi}_\infty
= \sup_{\mu\in L^1\colon\norm{\mu}_1 = 1} \norm{\mu Q}_1\;.
\end{equation} 
In particular, \eqref{eq:bound_Pij} implies that the kernel $R = K^\star - \id$ satisfies 
\begin{equation}
\label{eq:bound_normR} 
 \norm{R} \leqs 2(N-1) \e^{-(H_0-\eta)/\sigma^2}\;.
\end{equation} 
This allows us to bound the resolvent of $K^\star$. Indeed, for $z\in\C\setminus\set{1}$, 
we have 
\begin{equation}
 (z\id - K^\star)^{-1} 
 = \frac{1}{z-1} \biggpar{\id - \frac{1}{z-1}R}^{-1}
 = \frac{1}{z-1} \sum_{n\geqs0} \frac{1}{(z-1)^n}R^n\;.
\end{equation}
The Neumann series converges whenever $\abs{z-1} > \norm{R}$, in which case we have 
\begin{equation}
\label{eq:resolvent_Kstar} 
 \norm{(z\id - K^\star)^{-1}} \leqs \frac{1}{\abs{z-1} - \norm{R}}\;.
\end{equation} 
It follows that all eigenvalues of $K^\star$ are contained in the closed disc of radius $\norm{R}$ centred in $1$. 

Our aim is now to compare $K^u$ and $K^\star$ in two steps. Firstly, \cite[Prop.~6.1]{BaudelBerglund_2017}, 
slightly adapted to allow for complex $u$, shows that for any $m\in\N$, 
\begin{equation}
\label{eq:bound_KuK0} 
 \norm{(K^u)^m - (K^0)^m}
 \leqs \Biggpar{1 + \frac{\abs{1 - \e^{-u}} \expecin{\MN}{\tau^+_{\MN}-1}}
 {1 - \abs{1 - \e^{-u}}\expecin{\MN^c}{\tau^+_\MN}}}^m - 1\;,
\end{equation} 
which holds provided the denominator is strictly positive. This is indeed the case provided 
$\abs{1 - \e^{-u}} < E_\X(\sigma)^{-1}$. Secondly, \cite[Prop.~6.7]{BaudelBerglund_2017} 
shows that for any $m\in\N$, 
\begin{equation}
 \norm{(K^0)^m - (K^\star)^m}
 \leqs \sup_{1\leqs i\leqs N} R_i\;,
\end{equation} 
where
\begin{align}
 R_i ={}& \norm{\mathring\phi^{B_i}_0 - 1} 
 + 2 \bigabs{\mathring\lambda^{B_i}_1}^m 
 + 2 \frac{1 - \bigabs{\mathring\lambda^{B_i}_1}^m}{1 - \bigabs{\mathring\lambda^{B_i}_1}}
 \probin{B_i}{\tau^+_{\MN\setminus B_i} < \tau^+_{B_i}} \\
 &{}+ m(m-1) \probin{B_i}{\tau^+_{\MN\setminus B_i} < \tau^+_{B_i}}
 \probin{\MN\setminus B_i}{\tau^+_{B_i} < \tau^+_{\MN\setminus B_i}}\;.
 \label{eq:Ri} 
\end{align} 
Here the $\mathring\phi^{B_i}_k$ and $\mathring\lambda^{B_i}_k$ denote eigen-elements of the trace process on $\MN$ killed upon leaving $B_i$. These can be estimated thanks to the uniform positivity condition~\ref{ass:POS}. First note that integrating~\eqref{eq:UPC} against $\phi^A_0(y)$, yields the very rough bound 
\begin{equation}
\label{eq:oscillation_rough} 
 \sup_{x\in A} \phi^A_0 \leqs L \inf_{x\in A} \phi^A_0(x)\;.
\end{equation} 
With the normalisation $\pi^A_0(\phi^A_0)=1$, this yields $L^{-1} \leqs 
\phi^A_0(x) \leqs L$ for all $x\in A$. A much sharper bound is then provided by 
the following estimates. 

\begin{proposition}[Spectral gap and oscillation of $\phi^A_0$]
Let $K_A$ be the kernel of the process killed upon leaving $A$. 
Assume its density $k_A$ satisfies the uniform positivity condition~\eqref{eq:UPC} 
with parameters $n_0(\sigma)\in\N$ and $L\in(1,2)$. Then the spectral gap satisfies 
\begin{equation}
\label{eq:spectral_gap} 
 {\pth{\frac{\abs{\lambda^A_1}}{\lambda^A_0}}}^{n_0(\sigma)} 
 \leqs L - 
 \frac{\displaystyle \inf_{x\in A} \; \bigprobin{x}{\tau_{A^c} > n_0(\sigma)}}
{\bigpar{\lambda^A_0}^{n_0(\sigma)}}\;.
\end{equation} 
Furthermore, the oscillation of the principal eigenfunction satisfies
\begin{equation}
\label{eq:oscillation_phiA0} 
\norm{\phi^A_0 - 1}
:= \sup_{x\in A} \abs{\phi^A_0(x) - 1}
\leqs L^3 
\abs{1 - \frac{\displaystyle \inf_{x\in A} \; 
\bigprobin{x}{\tau_{A^c} > n_0(\sigma)}}
{\bigpar{\lambda^A_0}^{n_0(\sigma)}}}\;.
\end{equation} 
\end{proposition}
\begin{proof}
The spectral gap estimate~\eqref{eq:spectral_gap} is proved 
in~\cite[Prop.~5.1]{BaudelBerglund_2017}. The 
estimate \eqref{eq:oscillation_phiA0} is proved 
in~\cite[Prop.~5.5]{BaudelBerglund_2017}, where the constant $M$ in that 
result can be taken equal to $L$ thanks to the \textit{a priori}\/ bound  
$L^{-1} \leqs \phi^A_0(x) \leqs L$.
\end{proof}

We are going to apply these bounds to the kernel $\ptrace{K}{\MN}$ with $A = B_i$. 
Note that~\eqref{eq:ldp_MN_Bi} already allows us to bound several terms in~\eqref{eq:Ri} by exponentially small quantities. Furthermore, the definition of $\QSD{i}$ implies that for any $x\in B_i$, one has 
\begin{align}
 1 - \frac{\bigprobin{x}{\tau_{\MN\setminus B_i} > n_0(\sigma)}}
 {\bigpar{\mathring\lambda^{B_i}_0}^{n_0(\sigma)}} 
 &= 1 - \frac{\bigprobin{x}{\tau_{\MN\setminus B_i} > n_0(\sigma)}}
 {\bigprobin{\QSD{i}}{\tau_{\MN\setminus B_i} > n_0(\sigma)}} \\
 &= \frac{\bigprobin{x}{\tau_{\MN\setminus B_i} \leqs n_0(\sigma)}
 - \bigprobin{\QSD{i}}{\tau_{\MN\setminus B_i} \leqs n_0(\sigma)}}
 {1 - \bigprobin{\QSD{i}}{\tau_{\MN\setminus B_i} \leqs n_0(\sigma)}}\;.
\end{align}
Combining~\eqref{eq:oscillation_phiA0} and~\eqref{eq:ldp_MN_Bi} we get  
\begin{equation}
 \norm{\mathring\phi^{B_i}_0 - 1} 
 \leqs \frac{L^3 n_0(\sigma) \e^{-(H_0 - \eta)/\sigma^2}}{1 - \e^{-(H_0 + \eta)/\sigma^2}}\;.
\end{equation} 
A similar argument, based on the bound~\eqref{eq:spectral_gap}, shows that 
\begin{equation}
\label{eq:bound_lambda_QSD} 
 \bigabs{\mathring\lambda^{B_i}_1}
 \leqs 
 \Biggpar{\frac{L - 1 + n_0(\sigma)\e^{-(H_0-\eta)/\sigma^2}}
 {1 - \e^{-(H_0 + \eta)/\sigma^2}}}^{1/n_0(\sigma)}\;,
\end{equation}
Plugging the last two estimates into~\eqref{eq:Ri} yields 
\begin{align}
\norm{(K^0)^m - (K^\star)^m}
 \leqs&{} \frac{L^3 n_0(\sigma) \e^{-(H_0 - \eta)/\sigma^2}}{1 - \e^{-(H_0 + \eta)/\sigma^2}}
 + \Biggpar{\frac{L - 1 + n_0(\sigma)\e^{-(H_0-\eta)/\sigma^2}}
 {1 - \e^{-(H_0 + \eta)/\sigma^2}}}^{m/n_0(\sigma)} \\
 &{}+ 2m  \e^{-(H_0 - \eta)/\sigma^2} + m^2  \e^{-2(H_0 - \eta)/\sigma^2}\;,
\label{eq:bound_K0_Kstar} 
\end{align}
where we have bounded the fraction in~\eqref{eq:Ri} above by $m$. 


\subsection{Resolvent estimate}
\label{ssec:resolvent} 

We can now apply the following classical resolvent estimate, see for 
instance the argument presented in~\cite[Sect.~7.1]{BaudelBerglund_2017}, 
which is based on~\cite[Cor.~8.2]{gohberg2013basic} and 
\cite[Prop.~4.2]{gohberg2013classes}. 

\begin{lemma}
\label{lem:resolvent} 
Let $K_1$ and $K_2$ be compact linear operators. Let\/ $\Gamma$ be a contour in the complex plane, 
encircling $k$ eigenvalues of $K_1$. Let 
\begin{align}
\gamma &= \min\bigsetsuch{\norm{(z\id - K_1)^{-1}}^{-1}}{z\in\Gamma}\;, \\
C &= \frac{1}{\pi} \int_{\Gamma} \norm{(z\id - K_1)^{-1}}^2 \6z\;.
\end{align} 
If $\norm{K_2 - K_1} < \min\set{\frac12\gamma, C^{-1}}$, then $K_2$ has exactly $k$ 
eigenvalues inside the contour\/ $\Gamma$.
\end{lemma}

We now apply this lemma to $K_1 = (K^\star)^m$, and $K_2 = (K^u)^m$. The same argument as 
the one yielding~\eqref{eq:resolvent_Kstar} shows that   
\begin{equation}
 \norm{(z\id - (K^\star)^m)^{-1}} \leqs \frac{1}{\abs{z-1} + 1 - (1 + \norm{R})^m}\;.
\end{equation} 
It follows that all $N$ nonzero eigenvalues of $(K^\star)^m$ are contained in a disc of radius 
$(1 + \norm{R})^m - 1$, centred in $1$. Given $r \geqs 2[(1 + \norm{R})^m - 1]$, 
Lemma~\ref{lem:resolvent} applied to the contour $\Gamma$ of radius $r$, centred in $1$ 
shows that $(K^u)^m$ has exactly $N$ eigenvalues inside $\Gamma$, provided 
\begin{equation}
\label{eq:condition_norm} 
 \norm{(K^u)^m - (K^\star)^m} \leqs r\;.
\end{equation} 
We now make some convenient choices for various parameters. First of all,
we assume that 
\begin{equation}
 \eta \leqs \frac13\min\bigset{H_0, \theta_0}\;,
\end{equation} 
and take $\sigma$ small enough to guarantee that $n_0(\sigma) \leqs \e^{\eta/\sigma^2}$ and $E_\X(\sigma) \leqs \e^{\eta/\sigma^2}$, 
which is possible by \Cref{ass:POS} and Proposition~\ref{prop:ldp_EMN}.
Since $\delta = \frac12 (L - 1) > 0$, we may define 
\begin{equation}
 m 
 = \Biggintpartplus{\frac{1}{\sigma^2}
 \max\biggset{m_0, \frac{H_0 n_0(\sigma)}{\log(\delta^{-1})}}}\;.
 \label{eq:choice_m} 
\end{equation} 
Note that this implies $\delta^{m/n_0(\sigma)} \leqs \e^{-H_0/\sigma^2}$. 
Then it follows from~\eqref{eq:bound_K0_Kstar} that 
\begin{equation}
 \norm{(K^0)^m - (K^\star)^m} 
 \leqs \biggpar{1 + C_1\frac{n_0(\sigma)}{\sigma^2}}
 \e^{-(H_0-\eta)/\sigma^2}
\end{equation} 
for some numerical constant $C_1$. 
Next we note that for any $x\in\MN$, one has 
\begin{equation}
 \bigexpecin{x}{\tau^+_\MN - 1} 
 \leqs \bigprobin{x}{X_1 \in \MN^c} E_\X(\sigma) 
 \leqs \e^{-\theta_0/\sigma^2} E_\X(\sigma)\;.
\end{equation} 
If we further impose the condition
\begin{equation}
\label{eq:condition_r} 
 r \leqs \frac1{2 E_\X(\sigma)}\;,
\end{equation} 
then~\eqref{eq:bound_KuK0} yields 
\begin{equation}
  \norm{(K^u)^m - (K^0)^m}
 \leqs \bigpar{1 + 2r \expecin{\MN}{\tau^+_\MN - 1}}^m - 1
 \leqs \bigpar{1 + 2r\e^{-\theta_0/\sigma^2}E_\X(\sigma)}^m - 1
\end{equation} 
If $2mr\e^{-\theta_0/\sigma^2}E_\X(\sigma)$ is bounded, 
this quantity has order $mr\e^{-\theta_0/\sigma^2}E_\X(\sigma)$, so that 
\begin{equation}
 \norm{(K^u)^m - (K^0)^m}
 \leqs C_2 \frac{n_0(\sigma)E_\X(\sigma)}{\sigma^2} \e^{-\theta_0/\sigma^2}r 
 \leqs \frac{r}{2}
\end{equation} 
thanks to our bounds on $n_0(\sigma)$ and $E_\X(\sigma)$. We may thus set $r = \e^{-(H_0-2\eta)/\sigma^2}$, which satisfies both~\eqref{eq:condition_norm} and~\eqref{eq:condition_r}. One furthermore checks that the bound~\eqref{eq:bound_normR} implies that  $r \geqs 2[(1 + \norm{R})^m - 1]$. We can thus conclude that $K^u$ and $K^\star$ have the same number of eigenvalues in the disc $\set{\abs{z-1} < r}$. 

By a similar argument, $K^u$ has no eigenvalues in any contour that does not contain $0$, and stays at distance at least $r$ from $1$. It follows that $K^u$ has exactly $N$ nonzero eigenvalues (counting multiplicity). 

Recall finally that $K$ and $K^u$ have the same eigenvalues outside a disc of radius 
$\bigpar{\frac12}^{1/m_0}$. It follows that $K^m$ and $(K^u)^m$ have the same number of eigenvalues outside a disc of radius $\bigpar{\frac12}^{m/m_0}$. The choice~\eqref{eq:choice_m} of $m$ implies that this disc does not intersect the disc of radius $r$ centred in $1$, which concludes the proof. 
\qed


\section{Proof of Theorem~\ref{thm:reduction}}
\label{sec:proof_reduction} 


\subsection{Large-deviation estimates for the trace process on $\MN$}
\label{ssec:ldp_MN}

The trace process on $\MN$ is given by the sequence 
$\bigpar{X_{\tau^{+,n}_\MN}}_{n\in\N}$, where 
\begin{equation}
 \tau^{+,1}_\MN = \tau^+_\MN\;, \qquad 
 \tau^{+,n+1}_\MN = \inf\bigsetsuch{m > \tau^{+,n}_\MN}{X_m \in \MN}\;.
\end{equation} 
Owing to Proposition~\ref{prop:ldp_committor}, for any $\eta>0$ there exist 
$\sigma_0(\eta)>0$ and $\delta_0(\eta)>0$ such that  
\begin{equation}
\label{eq:LDP_Kstar} 
\e^{-[H(i,j)+\eta]/\sigma^2}
 \leqs \bigprobin{x}{X_{\tau^+_{\MN}} \in B_j}
 \leqs \e^{-[H(i,j)-\eta]/\sigma^2}
\end{equation} 
holds for all $x\in B_i$, $j\neq i$ and all $\sigma < \sigma_0(\eta)$, provided 
the diameter of the balls $B_\ell$ is bounded by $\delta_0(\eta)$.

We will use several properties of the quasipotential $H$ defined in~\eqref{eq:def_Hij}. First note that $H$ satisfies the triangle inequality 
\begin{equation}
\label{eq:triangle_H} 
 H(i,\ell) + H(\ell,j) \geqs H(i,j) 
 \qquad 
 \forall i, j, \ell \in \set{1,\dots,N}\;,
\end{equation} 
where we have extended $H$ by setting $H(i,i)=0$ for all $i\in\set{1,\dots,N}$.

We call \emph{path} a tuple $\gamma = (\gamma_0,\gamma_1,\dots,\gamma_{p-1},\gamma_p) \in \set{1,\dots,N}^{p+1}$ whose consecutive elements are different. Its \emph{length} is defined to be 
$\abs{\gamma} := p$, and its \emph{cost} is 
\begin{equation}
 V(\gamma) = H(\gamma_0,\gamma_1) + H(\gamma_1,\gamma_2) + \dots H(\gamma_{p-1},\gamma_p)\;.
\end{equation} 
We write $\gamma: i\to j$ if $\gamma_0 = i$ and $\gamma_p = j$. 
We say that $\gamma$ is an \emph{optimal path} from $i$ to $j$, and write $\gamma: i\optto j$, if 
\begin{equation}
 V(\gamma) = H(i,j)
\end{equation} 
and for any $\eps > 0$, there exists a sequence of points in the definition~\eqref{eq:def_Vxy} of the quasipotential that visits all $B_{\gamma_k}$ with $\gamma_k$ an element of $\gamma$, and whose cost is smaller than $H(i,j) + \eps$. Note that there may be more than one optimal path $\gamma: i\optto j$. 

We will also use the notation 
\begin{equation}
\label{eq:Hgap} 
 \Hgap = \min_{\gamma:i\to j, V(\gamma) > H(i,j)}
 \bigbrak{V(\gamma) - H(i,j)}
\end{equation} 
for the minimal difference between the costs of a non-optimal path and an optimal path from $i$ to $j$. The minimum is reached, even though the set of paths $\gamma:i\to j$ is infinite, because 
\begin{equation}
\label{eq:VH0} 
 V(\gamma) \geqs \abs{\gamma} H_0\;,
\end{equation} 
and an optimal path $\gamma: i\optto j$ can have length $N-1$ at most.  

\begin{proposition}
\label{prop:Px_LDP} 
For any $\eta>0$, there exist $\sigma_0(\eta), \delta_0(\eta) > 0$ and a constant $C_N$ depending only on $N$ such that 
\begin{align}
\bigprobin{x}{X_{\tau^{+,n}_{\MN}} \in B_j}
&\leqs \sum_{\gamma:i\optto j}
\binom{n}{\abs{\gamma}} \e^{-[H(i,j)-|\gamma|\eta]/\sigma^2}
+ C_N \e^{-[H(i,j)+\Hgap-N\eta]/\sigma^2}
\\
\bigprobin{x}{X_{\tau^{+,n}_{\MN}} \in B_j}
&\geqs \sum_{\gamma:i\optto j}
\binom{n}{\abs{\gamma}} \e^{-[H(i,j)+|\gamma|\eta]/\sigma^2}
\bigbrak{1 - \e^{-(H_0-\eta)/\sigma^2}}^{n-\abs{\gamma}}
\end{align}
holds for any $i\neq j$, $x\in B_i$ and $n\geqs1$, provided $\sigma < 
\sigma_0(\eta)$ and the diameter of the $B_k$ is bounded by $\delta_0(\eta)$.  
\end{proposition}
\begin{proof}
To any trajectory $(X_{\tau^{+,k}_{\MN}})_{0\leqs k\leqs n}$ from $x$ to $B_j$, 
we associate a path $\gamma = (i, \ell_1, \dots, \ell_p = j): i\to j$ and an increasing sequence $0 = k_0 < k_1 < k_2 < k_p \leqs n$ of jump times, such that 
\begin{equation}
 X_{\tau^{+,k}_{\MN}} \in B_{\gamma_\ell}
 \qquad \text{for $k_\ell \leqs k < k_{\ell+1}$}\;. 
\end{equation} 
The path $\gamma$ simply indicates the sequence of visited balls. Then we have 
\begin{equation}
\label{eq:proof_transition_n} 
 \bigprobin{x}{X_{\tau^{+,n}_{\MN}} \in B_j} 
 = \sum_{\substack{\gamma:i\to j\\ \abs{\gamma}\leqs n}}
 \; \sum_{0 < k_1 < \dots < k_p}
 Q_{k_1, \dots, k_p}(x)\;,
\end{equation} 
where we have set $p = \abs{\gamma}$, and 
\begin{equation}
 Q_{k_1, \dots, k_p}(x)
 = \bigprobin{x}{X_{\tau^{+,k}_{\MN}} \in B_{\gamma_\ell}, 
 k_\ell \leqs k < k_{\ell+1}, 0 \leqs \ell \leqs p-1}\;.
\end{equation} 
We want to use the fact that the sum~\eqref{eq:proof_transition_n} is dominated by optimal paths $\gamma:i\optto j$. 
Let $\gamma$ be such an optimal path, of length $p$. 
Then \eqref{eq:LDP_Kstar} yields the upper bound 
\begin{equation}
\label{eq:proof_transition_upper} 
 Q_{k_1, \dots, k_p}(x) 
 \leqs \prod_{\ell=0}^{p-1} 
 \sup_{y\in B_{\gamma_\ell}} \bigprobin{y}{X_{\tau^+_{\MN}} \in B_{\gamma_{\ell+1}}} 
 \leqs \prod_{\ell=0}^{p-1} 
 \e^{-[H(\gamma_\ell, \gamma_{\ell+1})-\eta]/\sigma^2}
 \leqs \e^{-[H(i,j)-p\eta]/\sigma^2}\;.
\end{equation} 
As a lower bound, we have 
\begin{align}
 Q_{k_1, \dots, k_p}(x)
&\geqs 
\prod_{\ell=0}^{p-1} 
\biggpar{\inf_{y\in B_{\gamma_\ell}} 
\bigprobin{y}{X_{\tau^+_{\MN}} \in B_{\gamma_{\ell+1}}}}
\prod_{\ell=0}^{p} 
\biggpar{\inf_{y\in B_{\gamma_\ell}} 
\bigprobin{y}{X_{\tau^+_{\MN}} \notin B_{\gamma_\ell}}}^{k_{\ell+1}-k_\ell-1}
\\
&\geqs \bigbrak{1 - \e^{-(H_0-\eta)/\sigma^2}}^{n-p} 
 \e^{-[H(i,j)+p\eta]/\sigma^2}\;.
\label{eq:proof_transition_lower} 
\end{align} 
Since both bounds are independent of the sequence of jump times 
$(k_1,\dots,k_p)$, summing over all these sequences simply multiplies the bounds by their number. This number is exactly the number of compositions of $n+1$ into $p+1$ parts, which is known to be equal to the binomial coefficient $\binom{n}{p}$.

It remains to bound the contribution of non-optimal paths. Here we distinguish the cases $\abs{\gamma}\leqs N$, and $\abs{\gamma} > N$. In the first case, we use~\eqref{eq:Hgap}, while in the second case, we use~\eqref{eq:VH0} and bound the resulting sum by a geometric series. The constant $C_N$ bounds the number of paths of length $N$, and can be taken of order $N^N$. 
\end{proof}


\subsection{The finite rank kernel $K^\star$}
\label{ssec:Kstar} 


In what follows, it will be convenient to use the physicists' bra-ket notation, 
in which a signed measure $\mu$, viewed as a row vector, is denoted 
$\bra{\mu}$, 
while a test function $f$, viewed as a column vector, is denoted $\ket{f}$. 
Recall that the kernel $K^\star$ is defined by 
\begin{equation}
\label{eq:Kstar} 
 K^\star(x,\6y) 
 = \sum_{i=1}^N \ind{x\in B_i} 
\bigprobin{\QSD{i}}{X_{\tau^+_{\MN}}\in \6y}\;.
\end{equation} 
Denote by $\cE^\star_\infty \subset L^\infty(\MN)$ its right image. This is 
an $N$-dimensional vector space, admitting the explicit basis 
\begin{equation}
 \cE^\star_\infty 
 = \vspan(\ket{\indicator{B_1}},\dots,\ket{\indicator{B_N}})\;.
\end{equation} 
In other words, $\cE^\star_\infty$ is the vector space of bounded measurable 
functions which are constant on each $B_i$. In particular, we have 
\begin{equation}
\label{eq:Kstar_braket} 
 \bigpar{K^\star\ket{\indicator{B_j}}}(x) 
 = \sum_{i=1}^N \ind{x\in B_i} \bigprobin{\QSD{i}}{X_{\tau^+_{\MN}}\in B_j}
 = \sum_{i=1}^N \ind{x\in B_i} 
\braketK{\QSD{i}}{K^0}{\indicator{B_j}}
\end{equation} 
where $K^0 = \ptrace{K}{\MN}$ denotes the trace of $K$ on $\MN$. This can 
be rewritten as 
\begin{equation}
\label{eq:KstarBj} 
 K^\star\ket{\indicator{B_j}} 
 = \sum_{i=1}^N \ket{\indicator{B_i}} \braketK{\QSD{i}}{K^0}{\indicator{B_j}}
 = \Pi^\star K^0 \ket{\indicator{B_j}}\;,
\end{equation} 
where 
\begin{equation}
\label{eq:Pistar} 
 \Pi^\star = \sum_{i=1}^N \ketbra{\indicator{B_i}}{\QSD{i}}
\end{equation} 
is the projector on $\cE^\star_\infty$; indeed, $(\Pi^\star)^2 = \Pi^\star$, 
owing to the orthonormality relation 
\begin{equation}
\label{eq:orthonormality} 
 \braket{\QSD{i}}{\indicator{B_j}} = \int_{B_j} \QSD{i}(\6x) = \delta_{ij}\;.
\end{equation}
Relation~\eqref{eq:KstarBj} shows that $K^\star = \Pi^\star K^0$ holds on 
$\cE^\star_\infty$.

Since $K^\star$ involves the QSDs $\QSD{i}$, it is natural to introduce the 
dual space 
\begin{equation}
 \cE^\star_1 
 = \vspan(\bra{\QSD{1}},\dots,\bra{\QSD{N}}) \subset L^1(\MN)\;.
\end{equation}
Note that the kernel $K^\star$, as defined in~\eqref{eq:Kstar}, does not 
necessarily leave $\cE^\star_1$ invariant. This is because even if $X_0$ is 
distributed according to the QSD $\bra{\QSDsmash{i}}$, conditionally on 
$X_1\in B_j$ with $j\neq i$, $X_1$ need not be distributed according to the QSD 
$\bra{\QSDsmash{j}}$. However, the kernel 
\begin{equation}
 \hat K^\star = K^\star \Pi^\star 
 = \sum_{j=1}^N K^\star \ketbra{\indicator{B_j}}{\QSD{j}}
 = \sum_{i,j=1}^N  \ket{\indicator{B_i}} 
\braketK{\QSD{i}}{K^0}{\indicator{B_j}} \bra{\QSD{j}}
\end{equation} 
does leave $\cE^\star_1$ invariant. The probabilistic interpretation of $\hat 
K^\star$ is that it acts as $K^\star$, but in addition it projects the 
distribution of $X_1$ on the QSD $\bra{\QSDsmash{j}}$ whenever $X_1\in B_j$. 
Note that $\Pi^\star K^\star = K^\star$, so that we have 
$(\hat K^\star)^n = (K^\star)^n \Pi^\star$ for all $n\in\N$. 
Since $\hat K^\star \ket{\indicator{B_j}} = K^\star 
\ket{\indicator{B_j}}$, \eqref{eq:KstarBj} implies 
\begin{equation}
\label{eq:matrix_element_Kstar} 
 \braketK{\QSD{i}}{\hat K^\star}{\indicator{B_j}}
 = \braketK{\QSD{i}}{K^\star}{\indicator{B_j}}
 = \braketK{\QSD{i}}{K^0}{\indicator{B_j}}
 = \bigprobin{\QSD{i}}{X_{\tau^+_{\MN}}\in B_j}\;,
\end{equation} 
showing that $\hat K^\star$, $K^\star$ and $K^0$ coincide when viewed as 
kernels acting on the invariant spaces $\cE^\star_\infty$ and $\cE^\star_1$. 

In what follows, we will be interested in processes in which time has been 
sped up by a factor $m = m(\sigma)$. These will involve the kernel 
\begin{equation}
 K^\star_m = \Pi^\star (K^0)^m \Pi^\star\;,
\end{equation} 
which satisfies 
\begin{equation}
\label{eq:Kstar_matrix_elements} 
 \braketK{\QSD{i}}{K^\star_m}{\indicator{B_j}}
 =\braketK{\QSD{i}}{\pth{K^0}^{m}}{\indicator{B_j}}
 = \bigprobin{\QSD{i}}{X_{\tau^{+,m}_{\MN}} \in B_j}\;.
\end{equation} 
The following large-deviation estimate is an immediate consequence of Proposition~\ref{prop:Px_LDP}. 

\begin{lemma}
\label{lem:ldp_msigma} 
Assume $m = m(\sigma)$ satisfies 
\begin{equation}
\label{eq:m_ldp} 
 \lim_{\sigma\to0} \sigma^2 \log m(\sigma) = \theta
\end{equation} 
for some $\theta \in (0,H_0)$. Let $p$ be the length of the longest optimal 
path $\gamma:i\optto j$, and let 
\begin{equation}
 H_\theta(i, j) = H(i, j) - p\theta\;.
\end{equation} 
Then for any $\eta>0$, there exist $\sigma_0(\eta), 
\delta_0(\eta) > 0$ such that 
\begin{equation}
\label{eq:ldp_taum} 
\e^{-(H_\theta(i,j)+\eta)/\sigma^2}
 \leqs \bigprobin{x}{X_{\tau^{+,m}_{\MN}} \in B_j} 
 \leqs \e^{-(H_\theta(i,j)-\eta)/\sigma^2}
\end{equation}
holds for any $i\neq j$ and $x\in B_i$, provided $\sigma < \sigma_0(\eta)$ and the diameter of the $B_k$ is bounded by $\delta_0(\eta)$. 
Furthermore, if\/ $(N-2)\theta \leqs \Hhat_0$ then $H_\theta$ satisfies the triangle inequality 
\begin{equation}
\label{eq:triangle_Htheta}
H_\theta(i,\ell) + H_\theta(\ell,j) \geqs H_\theta(i,j) 
\qquad \forall i, j, \ell\in\set{1,\dots,N}\;.
\end{equation} 
\end{lemma}
\begin{proof}
To show~\eqref{eq:ldp_taum}, it suffices to recall that all optimal paths have a length bounded by $N$. Therefore, the binomial coefficient $\smash{\binom{m}{p}}$ is logarithmically equivalent to $m^p$, so that the sum over optimal paths yields a prefactor equivalent to $\e^{p\theta/\sigma^2}$. 

To prove~\eqref{eq:triangle_Htheta}, we distinguish between two cases. 
If $H(i,\ell) + H(\ell,j) = H(i,j)$, then $\ell$ lies on an optimal path $\gamma:i\optto j$. This path is thus the concatenation of optimal paths $\gamma_1:i\optto \ell$ and $\gamma_2: \ell\optto j$, so that 
\begin{equation}
 H_\theta(i,j) = H(i,j) - \bigpar{\abs{\gamma_1} + \abs{\gamma_2}}\theta 
 = H_\theta(i,\ell) + H_\theta(\ell,j)\;.
\end{equation} 
The other possibility is that $H(i,\ell) + H(\ell,j) \geqs H(i,j) + \Hhat_0$. 
For optimal paths $\gamma:i\optto j$, $\gamma_1:i\optto \ell$ and $\gamma_2: \ell\optto j$, 
one gets 
\begin{equation}
 H_\theta(i,\ell) + H_\theta(\ell,j) - H_\theta(i,j) 
 \geqs \Hhat_0 - \bigpar{\abs{\gamma_1} + \abs{\gamma_2} - \abs{\gamma}} \theta\;. 
\end{equation} 
Since $\abs{\gamma_1} + \abs{\gamma_2} \leqs N-1$ and $\abs{\gamma} \geqs 1$, the result follows. 
\end{proof}


\subsection{The truncated kernel $\trunc{K^0}$}
\label{ssec:K0trunc} 

Denote by by $\lambda^0_k$, $\ket{\phi^0_k}$ and $\bra{\pi^0_k}$ the 
orthonormalised eigen-elements of $K^0$, and introduce the truncated kernel 
$\trunc{K^0}$ associated with the $N$ largest eigenvalues. 

We denote the right and left invariant subspaces of $\trunc{K^0}$ by 
\begin{align}
 \cE^0_\infty &= \vspan(\ket{\phi^0_0}, \dots, \ket{\phi^0_{N-1}})\;, \\
 \cE^0_1 &= \vspan(\bra{\pi^0_0}, \dots, \bra{\pi^0_{N-1}})\;. 
\end{align}
Our aim is now to construct another basis of the subspaces $\cE^0_\infty$ and 
$\cE^0_1$, which is close to the basis formed by the QSDs 
$\smash{\bra{\QSD{i}}}$ and the indicators $\ket{\indicator{B_j}}$. A natural 
idea is to set, for some $m\in\Z^*$, 
\begin{equation}
 \bra{\mu_i} = \bra{\QSD{i}} \pth{\trunc{K^0}}^m\;, 
 \qquad 
 \ket{\psi_j} = \pth{\trunc{K^0}}^{-m} \ket{\indicator{B_j}}\;, 
\end{equation} 
where $\pth{\trunc{K^0}}^{-1}$ is the generalised inverse of $\trunc{K^0}$, 
and $\pth{\trunc{K^0}}^{-m} = \pth{\pth{\trunc{K^0}}^{-1}}^m$. Indeed, we 
then have $\bra{\mu_i} \in \cE^0_1$ and $\ket{\psi_j} \in \cE^0_\infty$ by 
construction. Unfortunately, the basis is not orthonormal, because
$\pth{\trunc{K^0}}^{-1} \trunc{K^0} = \trunc{K^0} \pth{\trunc{K^0}}^{-1}
= \Pi^0$, where 
\begin{equation}
\label{eq:def_Pi0} 
 \Pi^0 = \sum_{k=0}^{N-1} \ketbra{\phi^0_k}{\pi^0_k} 
\end{equation}
is the projector on the invariant subspaces of $\trunc{K^0}$. 
Therefore, in general we will have 
\begin{equation}
 \braket{\mu_i}{\psi_j} = \braketK{\QSD{i}}{\Pi^0}{\indicator{B_j}}
 \neq \delta_{ij}\;,
\end{equation} 
A solution to this problem is to modify the definition of $\bra{\QSD{i}}$ and  
$\ket{\psi_j}$ as follows. 

\begin{lemma}
\label{lem:mupsi} 
Let $\Pi^0_\perp = \id - \Pi^0$. 
If\/ $\bra{\QSD{i}}\Pi^0\Pi^\star \neq 0$ for $j=1,\dots, N$, 
then the basis defined by 
\begin{equation}
 \bra{\mu_i} = \bra{\QSD{i}}  \bigbrak{\id - 
\Pi^0_\perp\Pi^\star}^{-1} \Pi^0\;, 
 \qquad 
 \ket{\psi_j} = \Pi^0
 \ket{\indicator{B_j}}
\end{equation} 
satisfies $\braket{\mu_i}{\psi_j} = \delta_{ij}$ for all 
$i,j\in\set{1,\dots,N}$.  Furthermore, 
\begin{equation}
\label{eq:mu_psi_complete} 
 \sum_{i=1}^N \ket{\psi_i}\bra{\mu_i} = \Pi^0\;.
\end{equation} 
\end{lemma}
\begin{proof}
Since $\bra{\QSD{i}}\Pi^\star = \bra{\QSD{i}}$, we have 
\begin{equation}
 \bra{\QSD{i}}\bigbrak{\id - \Pi^0_\perp\Pi^\star} 
 = \bra{\QSD{i}} \bigbrak{\id - \Pi^\star + \Pi^0\Pi^\star} 
 = \bra{\QSD{i}}\Pi^0\Pi^\star \neq 0
\end{equation} 
for $j=1,\dots,N$, so that $\bra{\mu_j}$ is indeed well-defined. Furthermore, 
\begin{align}
\braket{\mu_i}{\psi_j} 
&= \braketK{\QSD{i}}{\bigbrak{\id - \Pi^0_\perp\Pi^\star}^{-1} \Pi^0} 
{\indicator{B_j}} \\
&= \braketK{\QSD{i}}{\bigbrak{\id -\Pi^0_\perp\Pi^\star}^{-1}
\bigbrak{\id - \Pi^0_\perp}}{\indicator{B_j}} \\
&= \braketK{\QSD{i}}{\bigbrak{\id - \Pi^0_\perp\Pi^\star}^{-1} 
\bigbrak{\id - \Pi^0_\perp\Pi^\star}} 
{\indicator{B_j}} \\
&= \braket{\QSD{i}}{\indicator{B_j}} \\
&= \delta_{ij}\;,
\end{align}
where we have used the fact that $\Pi^\star\ket{\indicator{B_j}} = 
\ket{\indicator{B_j}}$ to obtain the third line.
As a consequence, the left-hand side of~\eqref{eq:mu_psi_complete} is a projector of rank $N$. Since it is of the form $\Pi^0 M \Pi^0$ for a linear operator $M$, its left and right images are given by $\cE^0_1$ and $\cE^0_\infty$, which implies~\eqref{eq:def_Pi0}. 
\end{proof}


\subsection{Comparison of transition probabilities}
\label{ssec:compare_ptrans} 

Our aim is now to show that for an appropriate $m = m(\sigma)\in\N$, one has 
\begin{equation}
\braketK{\mu_i}{\pth{\trunc{K^0}}^m}{\psi_j}
\sim
\braketK{\QSD{i}}{\pth{K^0}^m}{\indicator{B_j}}
= \bigprobin{\QSD{i}}{X_{\tau^{+,m}_{\MN}}\in B_j}\;.
\end{equation} 
The Neumann series representation of $[\id-\Pi^0_\perp\Pi^\star]^{-1}$ 
and the definition~\eqref{eq:Pistar} of $\Pi^\star$ yield   
\begin{align}
\braketK{\mu_i}{\pth{\trunc{K^0}}^m}{\psi_j}
&= \braketK{\QSD{i}}{\bigbrak{\id-\Pi^0_\perp\Pi^\star}^{-1} 
\pth{\trunc{K^0}}^m}{\indicator{B_j}} \\
&= \sum_{n\geqs0} \braketK{\QSD{i}}{
\pth{\Pi^0_\perp\Pi^\star}^n\pth{\trunc{K^0}}^m}{\indicator{B_j}} \\
&= \braketK{\QSD{i}}{\pth{\trunc{K^0}}^m}{\indicator{B_j}}
+ \sum_{\ell=1}^N \sum_{n\geqs1} 
\braketK{\QSD{i}}{\pth{\Pi^0_\perp\Pi^\star}^n}{\indicator{B_\ell}}
\braketK{\QSD{\ell}}{\pth{\trunc{K^0}}^m}{\indicator{B_j}}\;. 
\label{eq:matrix_element_K0m} 
\end{align}
For an appropriate $m$, the first term on the right-hand side is indeed close 
to $\braketK{\QSD{i}}{\pth{K^0}^m}{\indicator{B_j}}$. We thus need to bound 
the remaining terms. We introduce the notation
\begin{equation}
\label{eq:epsij_Pi} 
 \eps_{ij} = \eps^{(1)}_{ij} 
 = \braketK{\QSD{i}}{\Pi^0_\perp\Pi^\star}{\indicator{B_j}}
 = \braketK{\QSD{i}}{\Pi^\star - \Pi^0}{\indicator{B_j}}
 = \delta_{ij} - \braketK{\QSD{i}}{\Pi^0}{\indicator{B_j}}\;. 
 \end{equation} 
For every $n\geqs2$, \eqref{eq:Pistar} allows us to write 
\begin{equation}
 \eps^{(n)}_{ij} 
 = \braketK{\QSD{i}}{\pth{\Pi^0_\perp\Pi^\star}^n}{\indicator{B_j}}
 = \sum_{\ell_1,\dots,\ell_{n-1}=1}^N 
 \eps_{i\ell_1}\eps_{\ell_1\ell_2}\dots\eps_{\ell_{n-1}j}\;.
\end{equation} 
With these notations, \eqref{eq:matrix_element_K0m} becomes 
\begin{equation}
\label{eq:matrix_element_K0m:2} 
\braketK{\mu_i}{\pth{\trunc{K^0}}^m}{\psi_j}
= \braketK{\QSD{i}}{\pth{\trunc{K^0}}^m}{\indicator{B_j}}
+ \sum_{\ell=1}^N  
\braketK{\QSD{\ell}}{\pth{\trunc{K^0}}^m}{\indicator{B_j}}
\sum_{n\geqs1} \eps^{(n)}_{i\ell}\;. 
\end{equation}
The kernel $K^\star_m = \Pi^\star \pth{K^0}^m \Pi^\star$ has the same image as 
$\hat K^\star$ and $\Pi^\star$. Thus the Riesz projector formalism shows that 
 \begin{equation}
 \label{eq:Pistar_Pi0} 
 \Pi^\star - \Pi^0 
 = \frac{1}{2\pi\icx} \int_{\Gamma} 
 \bigbrak{\pth{z\id - K^\star_m}^{-1} 
 - \pth{z\id - \pth{\trunc{K^0}}^m}^{-1}} \6z \;,
 \end{equation}
provided $\Gamma$ is a contour in the complex plane encircling all (nonzero) 
eigenvalues of $K^\star$ and $\trunc{K^0}$. One option would be to use a 
resolvent identity and a bound on $\norm{K^\star-\trunc{K^0}}$, but this would 
yield an estimate which is uniform in $i, j$, which is not sharp enough for our 
purpose. 

To obtain a sharper bound, we note that the Cayley--Hamilton theorem 
implies that if $K$ is an operator of finite rank $N$, then $K^N$ is a linear 
combination of $\id, K, K^2, \dots, K^{N-1}$. This implies that the resolvent 
of $K$ can also be expressed in terms of a finite number of powers of $K$, as 
shows the following result. 

\begin{lemma}
\label{lem:resolvent_finite} 
Let $K$ be an operator of finite rank $N$, and let $\lambda_1,\dots,\lambda_N$ 
be its nonzero eigenvalues. Then the resolvent of $K$ can be written in the 
form 
\begin{equation}
\label{eq:resolvent_finite} 
 (z\id - K)^{-1} = \frac{1}{c_K(z)} \sum_{n=0}^{N-1} \alpha_n(z)K^n\;,
\end{equation} 
where $c_K(z) = \det(z\id - K) = \prod_{k=1}^N (z-\lambda_k)$ is the 
characteristic polynomial of $K$, and $\alpha_n(z)$ is a polynomial of degree 
$N-1-n$ in $z$. More precisely, one has 
\begin{equation}
 c_K(z) = \sum_{n=0}^N c_n z^n 
 \qquad \Rightarrow \qquad 
 \alpha_n(z) = \sum_{i=0}^{N-1-n} c_{i+n+1} z^i\;.
\end{equation} 
\end{lemma}
\begin{proof}
Multiply~\eqref{eq:resolvent_finite} by $(z\id - K)$ and use the relations 
$z\alpha_n(z) - \alpha_{n-1}(z) = -c_n$, $\alpha_{N-1}(z)=c_N$, 
$z\alpha_0(z)=c_K(z)-c_0$ and $c_K(K)=0$. Also see for instance~\cite{Hou98}, 
which gives an iterative construction implying the above expression for the 
$\alpha_n(z)$. 
\end{proof}

The key estimates that will allow us to control $\eps_{ij}$ are the following two propositions. They will allow us to control the error made when projecting the law of the process on its QSD every $m$ steps, and thus to compare matrix elements involving $\pth{K^0}^m$ and $K^\star_m$. This approach is somewhat related in spirit to the one used  in~\cite{Martinelli_Olivieri_Scoppola_89}. To lighten notations, we write 
\begin{equation}
\label{eq:def_rhoi}
\specgap{i} = \frac{\bigabs{\nextev{i}}}{\pev{i}}
\end{equation} 
for the spectral gap of the trace process killed upon leaving $B_i$, where $\nextev{i}$ is the next-to-leading eigenvalue of this process.

\begin{proposition}
\label{prop_Px_Ppi}
For any $\eta>0$, there exist $\sigma_0(\eta), \delta_0(\eta)>0$ such that 
\begin{equation}
\label{eq:Px_Ppi} 
 \bigprobin{x}{X_{\tau^{+,m}_{\MN}} \in B_j} 
 = \bigprobin{\QSD{i}}{X_{\tau^{+,m}_{\MN}} \in B_j}
 \bigbrak{1 + r_{\eta,m}(\sigma)}
\end{equation} 
holds for any $m\in \N$, any $i, j \in\set{1,\dots,N}$ and any $x\in B_i$, 
provided $\sigma <\sigma_0(\eta)$ and the diameter of the $B_\ell$ is bounded by 
$\delta_0(\eta)$. There exists a constant $C$, independent of $\sigma$, $m$ and 
$\eta$, such that the remainder in~\eqref{eq:Px_Ppi} satisfies 
\begin{equation}
\label{eq:Px_Ppi_error} 
 \bigabs{r_{\eta,m}(\sigma)} \leqs C 
 \Biggbrak{\specgap{i}^{m_1}
 + \frac{\binom{m}{p} - \binom{m-m_1}{p}}{\binom{m-m_1}{p}} 
 \frac{\e^{2p\eta/\sigma^2}}
 {\bigbrak{1-\e^{-(H_0-\eta)/\sigma^2}}^{m-p}}
 + \delta_{ij}m\e^{-(H_0-\eta)/\sigma^2}}
\end{equation} 
for any $m_1 < m$, where $p$ is the length of the longest optimal path $\gamma: i\optto j$. 
\end{proposition}
\begin{proof}
In the case $i=j$, the large-deviation principle shows that 
\begin{equation}
 \bigprobin{x}{X_{\tau^{+,m}_{\MN}} \in B_i} 
 = 1 - \bigOrder{m\e^{-(H_0-\eta)/\sigma^2}}\;, 
\end{equation} 
so that the result follows at once by integrating this relation against 
$\QSD{i}(x)$. 

In the case $i\neq j$, we consider first the case where the optimal path from $i$ to $j$ has length $1$. Then the decomposition~\eqref{eq:proof_transition_n}
of the transition probability involves only $m$ terms $Q_k(x)$, 
and the bounds~\eqref{eq:proof_transition_upper} and~\eqref{eq:proof_transition_lower} reduce to  
\begin{equation}
\label{eq:Qk_LDP} 
 \bigbrak{1-\e^{-(H_0-\eta)/\sigma^2}}^{m-1} \e^{-(H(i,j)+\eta)/\sigma^2}
 \leqs Q_k(x) 
 \leqs \e^{-(H(i,j)-\eta)/\sigma^2}
\end{equation} 
uniformly in $k$, which provides an upper bound on the ratio between the 
largest and smallest $Q_k(x)$. We now split the $Q_k(x)$ into \lq\lq bad\rq\rq\ and 
\lq\lq good\rq\rq\ terms, the good ones being those with $k\geqs m_1$ 
chosen sufficiently large that the process has time to relax to the QSD 
$\QSD{i}$ before making the transition to $B_j$. With this splitting, we have 
\begin{align}
\sum_{k=m_1+1}^m Q_k(x)
\leqs \sum_{k=1}^m Q_k(x) 
&\leqs \sum_{k=m_1+1}^m Q_k(x) + \sum_{k=1}^{m_1} Q_k(x) \\
&\leqs \sum_{k=m_1+1}^m Q_k(x) \biggbrak{1+ \frac{m_1}{m-m1} 
\frac{\e^{2\eta/\sigma^2}}{\bigbrak{1-\e^{-(H_0-\eta)/\sigma^2}}^{m-1}}}\;.
\end{align}
In order to derive a sharper estimate for the good $Q_k$, we rewrite them in 
the form 
\begin{equation}
 Q_k(x) = \int_{B_j} \int_{B_i} 
 \bigpar{\mathring{k}^{B_i}}^{k-1}(x,y) k^0(y,z) 
\bigpar{\mathring{K}^{B_j}}^{n-k}(z,B_j) \6y \6z\;.
\end{equation} 
Using the facts that  
\begin{align}
\bigpar{\mathring{k}^{B_i}}^{k-1}(x,y) 
&= \bigpar{\pev{i}}^{k-1} \QSD{i}(y) + 
\bigOrder{\bigabs{\nextev{i}}^{k-1}}\;, \\
\int_{B_i} \QSD{i}(x_1) 
\bigpar{\mathring{k}^{B_i}}^{k-1}(x_1,y) \6x_1 
&= \bigpar{\pev{i}}^{k-1} \QSD{i}(y)\;, 
\end{align}
we obtain 
\begin{equation}
 Q_k(x) = \int_{B_i} \QSD{i}(x_1) Q_k(x_1) \6x_1 
 \Bigbrak{1 + \bigOrder{\specgap{i}^{k-1}}}\;.
\end{equation} 
Collecting terms and bounding the contribution of non-optimal paths as in the proof of 
Proposition~\ref{prop:Px_LDP}, we get  
\begin{equation}
 \frac{\bigprobin{x}{X_{\tau^{+,m}_{\MN}} \in B_j}}
 {\bigprobin{\QSD{i}}{X_{\tau^{+,m}_{\MN}} \in B_j}}   
 = 1 + \bigOrder{\specgap{i}^{k-1}}
 + \BiggOrder{\frac{m_1}{m-m_1} 
\frac{\e^{2\eta/\sigma^2}}
{\bigbrak{1-\e^{-(H_0-\eta)/\sigma^2}}^{m-1}}}
\end{equation} 
as claimed. 
To extend the proof to optimal paths of length larger than $1$, we proceed in an 
analogous way, where the bad terms are those for which $k_1 \leqs m_1$. 
The ratio of binomial coefficients in~\eqref{eq:Px_Ppi_error} is the ratio between bad terms and good terms. 
\end{proof}

\begin{proposition}
\label{prop:expecF}
Fix $i\neq j$ and a bounded, measurable, real-valued test function $F$ supported in $B_j$. There exists a constant $C$ such that for any $x\in B_i$ and any $m_1 < m$, 
one has  
\begin{equation}
\label{eq:bound_EF} 
 \frac{\bigexpecin{x}{F(X_{\tau^{+,m}_{\MN}})}}
 {\bigprobin{x}{X_{\tau^{+,m}_{\MN}} \in B_j}}
 = \bigbrak{1-p_{\eta,m}(\sigma)} 
 \bigexpecin{\QSD{j}}{F} 
 \bigbrak{1 + q_{\eta,m}(\sigma)}
  + p_{\eta,m}(\sigma) \Bigexpecin{x}{
 \bigexpecin{X_{\tau^{+,m}_{\MN}}}{F}}\;,
\end{equation} 
where the error terms satisfy 
\begin{equation}
\label{eq:bounds_pq} 
 \bigabs{p_{\eta,m}(\sigma)} \leqs C 
 \frac{\binom{m}{p} - \binom{m-m_1}{p}}{\binom{m-m_1}{p}} 
 \frac{\e^{2p\eta/\sigma^2}}
{\bigbrak{1-\e^{-(H_0-\eta)/\sigma^2}}^{m-p}}\;, 
\qquad 
 \bigabs{q_{\eta,m}(\sigma)} 
 \leqs C \Bigbrak{\specgap{j}^{m_1} + m\e^{-(H_0-\eta)/\sigma^2}}\;.
\end{equation} 
Here $p$ is again the length of the longest optimal path $\gamma: i\optto j$. 
If $i = j$, then the bound~\eqref{eq:bound_EF} holds with 
$\bigabs{p_{\eta,m}(\sigma)} \leqs m\e^{-(H_0-\eta)/\sigma^2}$. 
\end{proposition}
\begin{proof}
For $i\neq j$, let us consider again the case $p=1$ first. 
We introduce the stopping time
\begin{equation}
 \tau = \inf\bigsetsuch{n > 0}{X_{\tau^{+,n}_{\MN}}\in B_j} - 1\;.
\end{equation} 
Since the optimal path from $i$ to $j$ visits $j$ only once, at the very end of the path, $\tau$ will be with overwhelming probability equal to the last time the process visits $\MN\setminus\set{j}$. Then we write 
\begin{equation}
\label{eq:decomp_Ek} 
 \bigexpecin{x}{F(X_{\tau^{+,m}_{\MN}})}
 = \sum_{k=0}^{m-m_1} E_k(x) + \sum_{k=m-m_1+1}^{m-1} E_k(x)\;,
\end{equation} 
where 
\begin{equation}
 E_k(x) = \bigexpecin{x}{\ind{\tau = k} F(X_{\tau^{+,k}_{\MN}})}\;.
\end{equation} 
Note that 
\begin{equation}
 \hat Q_k(x) = \bigprobin{x}{\tau = k} 
 = \bigprobin{x}{X_{\tau^{+,k}_{\MN}} \notin B_j, X_{\tau^{+,k+1}_{\MN}} \in B_j, \dots, X_{\tau^{+,m}_{\MN}} \in B_j}
\end{equation} 
has similar properties as $Q_k(x)$ in the previous proof. In particular, it again satisfies~\eqref{eq:Qk_LDP} owing to the large-deviation principle. In the same spirit as in the previous proof, we consider the $E_k(x)$ in the first sum in~\eqref{eq:decomp_Ek} as good terms, and those in the second sum as bad terms. 

To estimate the good terms, we write for $k\leqs m-m_1$ 
\begin{equation}
 E_k(x) 
 = \int_{B_j} \int_{\MN\setminus B_j} 
 \bigpar{\mathring{k}^{B_i}}^{k-1}(x,y) k^0(y,z) 
\bigpar{\mathring{K}^{B_j}}^{m-k}(z,F) \6y \6z\;, 
\end{equation}
where 
\begin{align}
 (\mathring{K}^{B_j})^{m-k}(z,F) 
 &= \bigexpecin{z}{F(X_{\tau^{+,m-k}_\MN})} \\
 &= (\pev{j})^{m-k} \bigexpecin{\QSD{j}}{F}
 \bigbrak{1 + \Order{\specgap{j}^{m-k}}}\;.
\end{align}
Since $\pev{j} = 1 - \Order{\e^{-(H_0-\eta)/\sigma^2}}$, 
it follows that 
\begin{equation}
 E_k(x) 
 = \bigprobin{x}{\tau = k-1}
 \bigexpecin{\QSD{j}}{F}
 \bigbrak{1 + \Order{\specgap{j}^{m-k}} +  \Order{m\e^{-(H_0-\eta)/\sigma^2}}}\;,
\end{equation} 
so that the sum of good terms satisfies 
\begin{equation}
 \sum_{k=0}^{m-m_1} E_k(x)
 = \bigprobin{x}{\tau \leqs m-m_1}
 \bigexpecin{\QSD{j}}{F}
 \bigbrak{1 + \Order{\specgap{j}^{m_1}} +  \Order{m\e^{-(H_0-\eta)/\sigma^2}}}\;.
\end{equation} 
This implies the result, with 
\begin{equation}
 p_{\eta,m}(\sigma)
 = \bigpcondin{x}{\tau > m-m_1}{ X_{\tau^{+,m}_{\MN}} \in B_j}
 = \frac{\sum_{k=m-m_1+1}^{m-1}\hat Q_k(x)}
 {\sum_{k=0}^{m-1} \hat Q_k(x)}\;,
\end{equation} 
which satisfies indeed the bound~\eqref{eq:bounds_pq} with $p=1$, thanks to~\eqref{eq:Qk_LDP}. The case of general $p$ then follows in a similar way, by counting the number of good and bad terms. 

In the case $i=j$, the result follows by distinguishing the cases where $X_{\tau^{+,k}_\MN}\in B_i$ for all $k\leqs m$, and the unlikely complementary event.  
\end{proof}

\begin{corollary}
\label{cor:eps_ij} 
For any $\eta>0$, there exist $\sigma_0(\eta), \delta_0(\eta)>0$ such that 
\begin{equation}
\Biggabs{\frac{\braketK{\QSD{i}}{\pth{K^0}^{nm}}{\indicator{B_j}}}
{\braketK{\QSD{i}}{\pth{K^\star_m}^n}{\indicator{B_j}}} - 1}
\leqs 
R_{\eta,m,n}(\sigma) 
:= \Bigbrak{1 + q_{\eta,m}(\sigma)+p_{\eta,m}(\sigma)r_{\eta,nm}(\sigma)}^{n-1}
-1
 \label{eq:eps_ij_prop} 
\end{equation} 
holds for all $\sigma<\sigma_0$ and all $i,j$, provided the diameter of the 
$B_k$ is bounded by $\delta_0(\eta)$. 
\end{corollary}
\begin{proof}
For $n=1$, the result follows from~\eqref{eq:Kstar_matrix_elements} 
with $R_{\eta,m,1}(\sigma) = 0$.
To prove the result for $n\geqs2$, we first observe that 
\begin{align}
\braketK{\QSD{i}}{\pth{K^0}^{nm}}{\indicator{B_j}} 
&= \sum_{\ell=1}^N 
\Bigexpecin{\QSD{i}}{\bigind{X_{\tau^{+,m}_{\MN}} \in B_\ell}
\bigprobin{X_{\tau^{+,m}_{\MN}}}{X_{\tau^{+,(n-1)m}_{\MN}} \in B_j}} \\ 
&= \sum_{\ell=1}^N  
\bigexpecin{\QSD{i}}{F_\ell(X_{\tau^{+,m}_{\MN}})}\;,
\label{eq:proof_prop_epsij_1} 
\end{align}
where
\begin{equation}
 F_\ell(x) = \ind{x\in B_\ell} \bigprobin{x}{X_{\tau^{+,(n-1)m}_{\MN}} \in B_j}\;.
\end{equation} 
Proposition~\ref{prop:expecF} shows that 
\begin{align}
 \bigexpecin{\QSD{i}}{F_\ell(X_{\tau^{+,m}_{\MN}})}
 ={}& \bigprobin{\QSD{i}}{X_{\tau^{+,m}_{\MN}} \in B_\ell} \\
 &{}\times 
 \biggbrak{(1-p_{\eta,m}) \bigexpecin{\QSD{\ell}}{F_\ell} 
 (1 + q_{\eta,m})
  + p_{\eta,m} \Bigexpecin{\QSD{i}}{
   \bigexpecin{X_{\tau^{+,m}_{\MN}}}{F_\ell}}}\;.
\end{align} 
Now we note that 
\begin{equation}
 \bigexpecin{\QSD{\ell}}{F_\ell} 
 = \bigprobin{\QSD{\ell}}{X_{\tau^{+,(n-1)m}_{\MN}} \in B_\ell}\;,
\end{equation} 
while Proposition~\ref{prop_Px_Ppi} implies that for any $x\in B_\ell$, one has 
\begin{equation}
 \bigexpecin{x}{F_\ell} 
 = \bigprobin{\QSD{\ell}}{X_{\tau^{+,(n-1)m}_{\MN}} \in B_\ell}
 \bigbrak{1 + r_{\eta,(n-1)m}}\;.
\end{equation} 
It follows that 
\begin{align}
 \bigexpecin{\QSD{i}}{F_\ell(X_{\tau^{+,m}_{\MN}})}
 &= \bigprobin{\QSD{i}}{X_{\tau^{+,m}_{\MN}} \in B_\ell}
  \bigprobin{\QSD{\ell}}{X_{\tau^{+,(n-1)m}_{\MN}} \in B_\ell} 
  \bigbrak{1+R_n} \\
 &= \braketK{\QSD{i}}{\pth{K^0}^{m}}{\indicator{B_\ell}}
 \braketK{\QSD{\ell}}{\pth{K^0}^{(n-1)m}}{\indicator{B_j}}
 \bigbrak{1+R_n}\;,
 \label{eq:bound_expec_F} 
\end{align}
where the remainder
\begin{equation}
 R_n = (1-p_{\eta,m})(1+q_{\eta,m})+p_{\eta,m}(1+r_{\eta,(n-1)m}) - 1
\end{equation} 
satisfies 
\begin{equation}
 0 \leqs R_n \leqs q_{\eta,m} + p_{\eta,m}r_{\eta,(n-1)m}\;.
\end{equation} 
Summing~\eqref{eq:bound_expec_F} over $\ell$ shows that 
\begin{equation}
 \braketK{\QSD{i}}{\pth{K^0}^{nm}}{\indicator{B_j}} 
 = \braketK{\QSD{i}}{\pth{K^0}^{m}\Pi^\star\pth{K^0}^{(n-1)m}}{\indicator{B_j}} 
 \bigbrak{1+R_n}\;,
\end{equation} 
and the result follows by induction on $n$. 
\end{proof}

\begin{corollary}
\label{cor:epsij} 
Assume $m$ satisfies~\eqref{eq:m_ldp}. Then there exists a constant $C$ such that for any $\eta>0$ 
\begin{equation}
\label{eq:bound_epsij} 
 \bigabs{\eps_{ij}} \leqs C \Bigbrak{\delta_{ij} \e^{-(H_0-\eta)/\sigma^2} 
+ N\e^{-[H_\theta(i,j)-\eta]/\sigma^2}
 \bigbrak{R_{\eta,m,N}(\sigma) + \varrho^m + \e^{-(H_0-\eta)/\sigma^2}}}
\end{equation} 
holds for $1 \leqs i, j \leqs N$, provided $\sigma$ and the $B_k$ are 
sufficiently small as a function of $\eta$. 
\end{corollary}
\begin{proof}
Let $\Gamma$ be a contour encircling the (nonzero) eigenvalues of $(\trunc{K^0})^m$, and staying at a distance of order $1$ from all eigenvalues of $(K^0)^m$. Note that this is possible for $\sigma$ small enough by Proposition~\ref{prop:spectral_gap}. Recall that $\trunc{K^0} = \Pi^0 K^0$, where $\Pi^0$ is the Riesz projector associated with $\Gamma$. Using~\eqref{eq:epsij_Pi}, \eqref{eq:Pistar_Pi0}, Proposition~\ref{lem:resolvent_finite} and Corollary~\ref{cor:eps_ij}, we obtain  
\begin{align}
 \eps_{ij} 
 ={}& \sum_{n=0}^{N-1} \frac{1}{2\pi\icx} \int_{\Gamma} 
 \braketK{\QSD{i}}
 {\biggbrak{\frac{\alpha^\star_n(z)}{c^\star(z)} \pth{K^\star_m}^n 
 - \frac{\alpha^0_n(z)}{c^0(z)} \pth{\trunc{K^0}}^{nm}}}
 {\indicator{B_j}} \6z \\
 ={}& \sum_{n=1}^{N-1} \frac{1}{2\pi\icx} \int_{\Gamma}
 \biggbrak{\frac{\alpha^\star_n(z)}{c^\star(z)} 
 \bigbrak{1 + \Order{R_{\eta,m,n}(\sigma)}}
 - \frac{\alpha^0_n(z)}{c^0(z)}\bigbrak{1 + \Order{\varrho^{nm}}}} \6z \;
 \braketK{\QSD{i}}{\pth{K^0}^{nm}}{\indicator{B_j}} \\
 &{}+ \frac{1}{2\pi\icx} \int_{\Gamma}
 \biggbrak{\frac{\alpha^\star_0(z)}{c^\star(z)} 
 - \frac{\alpha^0_0(z)}{c^0(z)}} \6z \; \delta_{ij}\;, 
\label{eq:Riesz_epsij} 
\end{align}
where $c^\star(z)$ and the $\alpha^\star_n(z)$ are the coefficients of the decomposition~\eqref{eq:resolvent_finite} of $K^\star_m$, and $c^0(z)$ and the $\alpha^0_n(z)$ are those of the decomposition of $(\trunc{K^0})^m$. The contour $\Gamma$ has been chosen such that the characteristic polynomials $c^\star(z)$ and $c^0(z)$ are bounded away from $0$.  Proposition~\ref{prop:spectral_gap} shows that the eigenvalues of $K^\star$ and $K^0$ are at distance $\Order{\e^{-(H_0-\eta)/\sigma^2}}$ from each other. This shows that $\alpha^\star_n(z)/c^\star(z) = (\alpha^0_n(z)/c^0(z)) [1+\Order{\e^{-(H_0-\eta)/\sigma^2}}]$ on the contour $\Gamma$. 
Hence the result follows from Lemma~\ref{lem:ldp_msigma}. 
\end{proof}

We can now choose a value of $m_1$ yielding the smallest possible error 
terms. Since $N$ is a finite constant, we no longer indicate the dependence 
of the error terms on $N$. The bound~\eqref{eq:bound_lambda_QSD} implies that 
\begin{equation}
 \bigabs{\nextev{i}} \leqs \delta^{1/n_0(\sigma)}
 = \exp\biggset{-\frac{\log(\delta^{-1})}{n_0(\sigma)}}
\end{equation} 
for a constant $\delta < 1$ related to $L$. 
The choice 
\begin{equation}
 m_1 = \frac{H_0n_0(\sigma)}{\log(\delta^{-1})\sigma^2} 
\end{equation} 
then yields
\begin{equation}
\label{eq:bound_sgap} 
\specgap{i}^{m_1} \leqs
 \frac{\bigabs{\nextev{i}}^{m_1}}{\bigpar{\pev{i}}^{m_1}} \leqs 2\e^{-H_0/\sigma^2}\;.
\end{equation} 
Since $m$ is assumed to satisfy~\eqref{eq:m_ldp}, we have 
$2m_1 \leqs m \leqs \e^{(H_0-\eta)/\sigma^2}$ for $\eta$ small enough. 
Further note that for these $m_1$ and $p$ or order $1$, 
\begin{equation}
 \frac{\binom{m}{p} - \binom{m-m_1}{p}}{\binom{m-m_1}{p}} 
 = \biggOrder{p\frac{m_1}{m}}\;.
\end{equation} 
Substituting in~\eqref{eq:Px_Ppi_error} yields 
\begin{equation}
 \bigabs{r_{\eta,m}(\sigma)}
 \leqs C_1 \biggbrak{\e^{-H_0/\sigma^2} 
 + \frac1m \frac{n_0(\sigma)}{\sigma^2}
 \e^{2\eta/\sigma^2}}
 \leqs 2C_1 \e^{-(\theta-3\eta)/\sigma^2}\;. 
\end{equation} 
The error term $q_{\eta,m}(\sigma)$ also satisfies~\eqref{eq:bound_sgap}, 
while $p_{\eta,m}(\sigma)$ is at most of order $r_{\eta,m}(\sigma)$. 
Therefore, 
\begin{equation}
 \bigabs{R_{\eta,m,N}(\sigma)}
 \leqs C_2 \e^{-2(\theta-3\eta)/\sigma^2}\;. 
\end{equation} 
Furthermore, Proposition~\ref{prop:spectral_gap} implies that for these $m$, 
$\varrho^m$ is negligible with respect to $r_{\eta,m}(\sigma)$, provided 
$2\eta < \theta$. Writing 
\begin{equation}
 \Hhat_\theta(i,j) = H_0\delta_{ij} + H_\theta(i,j)(1 - \delta_{ij})\;,
\end{equation} 
we can rewrite the bound~\eqref{eq:bound_epsij} as 
\begin{equation}
 \bigabs{\eps_{ij}} 
 \leqs \e^{-(\Hhat_\theta(i,j) + \theta - 4\eta)/\sigma^2}\;.
\end{equation} 
Note that the fact that the error $R_{\eta,m,N}$ involves the product 
$p_{\eta,m}r_{\eta,m}$ instead of only one of these terms has improved the 
accuracy of the approximation.

\begin{remark}
Corollary~\ref{cor:epsij} shows in particular that the assumption that 
$\bra{\QSD{i}}\Pi^0\Pi^\star \neq 0$, made in Lemma~\ref{lem:mupsi}, is 
satisfied for small enough $\sigma$. Indeed, writing $\Pi^0$ as a contour integral 
as in the proof of the Corollary shows that 
$\braketK{\QSDsmash{i}}{\Pi^0}{\indicator{B_j}} = \delta_{ij} + 
\smash{\Order{\e^{-[H_0-\eta]/\sigma^2}}}$. 
Therefore, $\bra{\QSDsmash{i}}\Pi^0\Pi^\star = 
\sum_{j=1}^N\braketK{\QSDsmash{i}}{\Pi^0}{\indicator{B_j}}\bra{\QSDsmash{j}}$ 
is exponentially close to $\bra{\QSDsmash{i}}$. 
\end{remark}

\begin{corollary}
\label{cor:matrix_element_K0m} 
For any $\eta\in(0,\theta)$, there exist $\sigma_0, \delta_0 > 0$ such that 
for all $1\leqs i\neq j\leqs N$, 
\begin{equation}
 \braketK{\mu_i}{\pth{\trunc{K^0}}^m}{\psi_j}
 = \braketK{\QSD{i}}{\pth{K^0}^m}{\indicator{B_j}}
  \bigbrak{1 + \bigOrder{\e^{-(\theta - \eta)/\sigma^2}}}\;,
\end{equation} 
provided $\sigma < \sigma_0$ and all $B_k$ have a diameter smaller than $\delta_0$. 
\end{corollary}
\begin{proof}
Proceeding by induction on $n$ and using the triangle 
inequality~\eqref{eq:triangle_Htheta}, one easily obtains the bounds 
\begin{equation}
 \bigabs{\eps_{ij}^{(n)}} \leqs \e^{-[\Hhat_\theta(i,j)+n(\theta-4\eta)]/\sigma^2}
 \quad \forall n\geqs 1\;, 
 \qquad \qquad 
 \sum_{n\geqs 1}\bigabs{\eps_{ij}^{(n)}} \leqs 
2\e^{-[\Hhat_\theta(i,j)+\theta-4\eta]/\sigma^2}\;. 
\label{eq:bound_epsijn} 
\end{equation} 
It follows that the remainder in~\eqref{eq:matrix_element_K0m:2} satisfies  
\begin{equation}
\biggabs{\sum_{\ell=1}^N  
\braketK{\QSD{i}}{\pth{\trunc{K^0}}^m}{\indicator{B_\ell}}
\sum_{n\geqs1} \eps^{(n)}_{\ell j}} 
\leqs 2N \e^{-[\Hhat_\theta(i,j)+\theta-5\eta]/\sigma^2}\;,
\end{equation} 
provided $\sigma$ and the $B_k$ are sufficiently small, depending on $\eta$. 
On the other hand, we have 
\begin{equation}
 \braketK{\QSD{i}}{\pth{\trunc{K^0}}^m}{\indicator{B_j}}
 = \braketK{\QSD{i}}{\pth{K^0}^m + \Order{\varrho^m}}{\indicator{B_j}}
 \geqs \e^{-[\Hhat_\theta(i,j)+\eta]/\sigma^2} + \Order{\varrho^m}\;.
\end{equation} 
Proposition~\ref{prop:spectral_gap} shows that for $m$ 
as above, the error term $\Order{\varrho^m}$ is indeed negligible, yielding the 
claimed exponentially small multiplicative error, after redefining $\eta$. 
\end{proof}

The following result shows in which sense the new basis vectors $\bra{\mu_i}$ 
and $\ket{\psi_j}$ are close to $\bra{\QSD{i}}$ and $\ket{\indicator{B_j}}$. 

\begin{proposition}
\label{prop:norms_mupsi} 
The basis vectors satisfy 
\begin{equation}
\label{eq:basis_1} 
 \braket{\mu_i}{\indicator{B_j}} = \delta_{ij}
 \qquad \text{and} \qquad 
 \braket{\QSD{i}}{\psi_j} = \delta_{ij} - \eps_{ij}
\end{equation} 
for all $1\leqs i, j\leqs N$. 
Furthermore, for any $\eta > 0$, one has 
\begin{equation}
\label{eq:basis_2} 
  \norm{\psi_j - \indicator{B_j}}_\infty
 = \sup_{x\in\MN} \bigabs{\psi_j(x) - \indicator{B_j}(x)} 
 \leqs \e^{-[\Hhat_j-\eta]/\sigma^2}
\end{equation}
provided $\sigma$ and the diameters of the $B_i$ are small enough, 
where
\begin{equation}
\label{eq:def_Hhat_j} 
 \Hhat_j = \min_{i\neq j} \biggbrak{H(i,j) - \max_{\gamma: i\optto j}\abs{\gamma}}
 \geqs H_0 -(N-1)\theta\;.
\end{equation} 
\end{proposition}
\begin{proof}
The relations~\eqref{eq:basis_1} follow immediately from the definitions, since 
\begin{align}
 \braket{\mu_i}{\indicator{B_j}} 
 &= 
\braketK{\QSD{i}}{\bigbrak{\id-\Pi^0_\perp\Pi^\star}^{-1}\Pi^0}{\indicator{B_j}}
= \braket{\mu_i}{\psi_j} = \delta_{ij}\;, \\
\braket{\QSD{i}}{\psi_j} 
&= \braketK{\QSD{i}}{\Pi^0}{\indicator{B_j}}
= \delta_{ij} - \eps_{ij}\;.
\end{align}
In order to prove~\eqref{eq:basis_2}, we proceed as in the proof of 
Corollary~\ref{cor:epsij}, writing for $x\in B_i$ 
\begin{align}
  \psi_j(x)-\delta_{ij} 
 &= \braketK{\delta_x}{\Pi^0 - \Pi^\star}{\indicator{B_j}} \\ 
 &= \sum_{n=0}^{N-1} \frac{1}{2\pi\icx} \int_{\Gamma} 
 \braketK{\delta_x}
 {\biggbrak{\frac{\alpha^0_n(z)}{c^0(z)} \pth{\trunc{K^0}}^{nm}
 - \frac{\alpha^\star_n(z)}{c^\star(z)} \pth{K^\star_m}^n}}
 {\indicator{B_j}} \6z\;.
\end{align}
Propositions~\ref{prop:spectral_gap} and~\ref{prop_Px_Ppi} imply that 
for $0\leqs n\leqs N-1$, one has 
\begin{align}
 \braketK{\delta_x}{\pth{\trunc{K^0}}^{nm}}{\indicator{B_j}}
 &= \bigprobin{x}{X_{\tau^{+,nm}_{\MN}}\in B_j} + \Order{\varrho^{nm}} \\
 &= \bigprobin{\QSD{i}}{X_{\tau^{+,nm}_{\MN}}\in B_j} \bigbrak{1 + 
 r_{\eta,m}(\sigma)} + \Order{\varrho^{nm}}\;, 
\end{align} 
while the definition of $K^\star_m$ implies 
\begin{equation}
 \braketK{\delta_x}{\pth{K^\star_m}^n}{\indicator{B_j}}
 = \braketK{\QSD{i}}{\pth{K^\star_m}^n}{\indicator{B_j}}
 = \bigprobin{\QSD{i}}{X_{\tau^{+,nm}_{\MN}}\in B_j} \bigbrak{1 + 
 r_{\eta,m}(\sigma)}\;.
\end{equation}
Substituting, we find 
\begin{equation}
 \bigabs{\psi_j(x)-\delta_{ij}}
 = \bigOrder{\e^{-[\Hhat_\theta(i,j)-\eta]/\sigma^2}}\;.
\end{equation} 
The expression~\eqref{eq:def_Hhat_j} of $\Hhat_j$ follows from the definition of 
$\Hhat_\theta(i,j)$, and the fact that $\Hhat(i,i) = H_0$.
\end{proof}

\begin{remark}
Getting an $L^1$-estimate on the difference 
$\bra{\mu_i} - \bra{\QSD{i}}$ would require a sharper, pointwise estimate 
on densities, than in Corollary~\ref{cor:eps_ij}. Indeed, we have 
\begin{align}
\bra{\mu_i} - \bra{\QSD{i}}
&= \bra{\QSD{i}} \Pi^0 - \bra{\QSD{i}} + \sum_{n\geqs1} \bra{\QSD{i}} 
\pth{\Pi^0_\perp\Pi^\star}^n \Pi^0 \\
&= \bra{\QSD{i}} \bigbrak{\Pi^0-\Pi^\star} + \sum_{j=1}^N 
\sum_{n\geqs1} \eps_{ij}^{(n)} \bra{\QSD{j}} \Pi^0\;. 
\end{align}
Using~\eqref{eq:bound_epsijn}, one can bound the $L^1$-norm of the double sum 
by an exponentially small term. However, estimating the $L^1$-norm of 
$\bra{\QSD{i}} \bigbrak{\Pi^0-\Pi^\star}$ would require a pointwise estimate 
of 
\begin{equation}
 \bra{\QSD{i}} \bigbrak{\pth{K^0}^{nm} - \pth{K^\star_m}^n}
\end{equation} 
instead of an integral estimate as in Corollary~\ref{cor:eps_ij}.
\end{remark}


\subsection{Proof of the main approximation result}
\label{ssec:approx_main} 

Let $m$ be as in the previous section. 
Define a matrix $P$ of dimension $N\times N$ with elements 
\begin{equation}
\label{eq:approx_01} 
 P_{ij} = \braketK{\mu_i}{\pth{\trunc{K^0}}^m}{\psi_j}\;.
\end{equation} 
Corollary~\ref{cor:matrix_element_K0m} shows that 
\begin{equation}
\label{eq:P_ij} 
 P_{ij} = 
 \bigprobin{\QSD{i}}{X_{\tau^{+,m}_{\MN}}\in B_j}
 \bigbrak{1+\bigOrder{\e^{-(\theta - \eta)/\sigma^2}}}\;.
\end{equation} 

\begin{lemma}
\label{lem:P_stoch} 
$P$ is a stochastic matrix for sufficiently small $\sigma$. 
\end{lemma}
\begin{proof}
First note that $\sum_{j=1}^N \ket{\indicator{B_j}}=\ket{\phi^0_0}$, since both 
are identically equal to $1$. It follows that 
\begin{equation}
 \sum_{j=1}^N \ket{\psi_j} = \Pi^0 \sum_{j=1}^N \ket{\indicator{B_j}}
 = \Pi^0 \ket{\phi^0_0} = \ket{\phi^0_0}\;,
\end{equation} 
and thus 
\begin{equation}
\sum_{j=1}^N P_{ij} 
= \braketK{\mu_i}{\pth{\trunc{K^0}}^m}{\phi^0_0} 
= \braket{\mu_i}{\phi^0_0} 
= \sum_{j=1}^N \braket{\mu_i}{\psi_j} 
= 1\;. 
\end{equation}
Furthermore, \eqref{eq:P_ij} implies that the $P_{ij}$ are positive 
if $\sigma$ is small enough. Note that, as shown in Corollary~\ref{cor:matrix_element_K0m}, 
this is a consequence of the fact that $m$ diverges as $\sigma$ goes to $0$, owing to~\eqref{eq:m_ldp}.
\end{proof}

Let $(Y_n)_{n\geqs 0}$ be the Markov chain with transition matrix $P$. Then Theorem~\ref{thm:reduction} follows directly from Theorem~\ref{thm:approx} below. Here expectations and probabilities with respect to a signed measure are interpreted as differences of these quantities with respect to the positive and negative parts of that measure. 

\begin{theorem}
\label{thm:approx} 
If $X_n$ starts with the (signed) distribution $\mu_i$, then 
\begin{equation}
\label{eq:thm_reduc1} 
 \Bigexpecin{\mu_i}{\psi_j \pth{X_{\tau^{+,nm}_{\MN}}}}
 = \bigprobin{i}{Y_n = j}
\end{equation} 
holds for all $n\geqs0$ and all $j\in\set{1,\dots,N}$. As a consequence, 
\begin{equation}
\label{eq:thm_reduc2} 
 \bigprobin{\mu_i}{X_{\tau^{+,nm}_{\MN}} \in B_j} 
 = \bigprobin{i}{Y_n = j}\bigbrak{1 + \bigOrder{\e^{-[\Hhat_j-\eta]/\sigma^2}}}
 + \bigprobin{i}{Y_n \neq j}  \bigOrder{\e^{-[\Hhat_j-\eta]/\sigma^2}}
\end{equation}
for any $\eta > 0$, provided $\sigma$ and the $B_i$ are small enough. 
Furthermore, for all $x\in B_i$, one has 
\begin{equation}
\label{eq:thm_reduc3} 
 \bigprobin{x}{X_{\tau^{+,nm}_{\MN}} \in B_j} 
 = \bigprobin{i}{Y_n = j} + \bigOrder{\e^{-[\Hhat_{\min} - \eta])/\sigma^2}}
 + \bigOrder{\varrho^{nm}}\;,
\end{equation} 
where $\Hhat_{\min} = \min_\ell\Hhat_\ell \geqs H_0 - (N-1)\theta$. 
\end{theorem}
\begin{proof}
The first claim~\eqref{eq:thm_reduc1} follows from~\eqref{eq:approx_01} by 
taking the $n$th power of $P$, and using the completeness relation~\eqref{eq:mu_psi_complete}.
The second claim~\eqref{eq:thm_reduc2} is a 
consequence of the decomposition 
\begin{equation}
 \bigprobin{i}{Y_n = j} 
 = \int_{B_j} \bigprobin{\mu_i}{X_{\tau^{+,nm}_{\MN}} \in \6y} \psi_j(y) 
 + \sum_{\ell\neq j}
 \int_{B_\ell} \bigprobin{\mu_i}{X_{\tau^{+,nm}_{\MN}} \in \6y} 
\psi_j(y)\;.
\end{equation}
Indeed, writing $P^n_{ij}$ for the left-hand side and 
$Q^n_{ij} = \bigprobin{\mu_i}{X_{\tau^{+,nm}_{\MN}} \in B_j}$, 
Proposition~\ref{prop:norms_mupsi} yields 
\begin{equation}
 P^n_{ij} = \sum_{\ell = 1}^N Q^n_{i\ell} \bigbrak{\delta_{lj} + r_{lj}}\;,
\end{equation} 
where $r_{lj} = \Order{\e^{-[\Hhat_j-\eta]/\sigma^2}}$ for all $\ell$. This is equivalent to the matrix equation $P^n = Q^n[\id + R]$, which can be inverted using the Neumann series for $[\id + R]^{-1}$. The resulting expression of $Q^n$ in terms of $P^n$ and $R$ is equivalent to~\eqref{eq:thm_reduc2}.

To obtain~\eqref{eq:thm_reduc3}, we write 
\begin{align}
\bigprobin{x}{X_{\tau^{+,nm}_{\MN}} \in B_j}
&= \braketK{\delta_x}{(K^0)^{nm}}{\indicator{B_j}} \\
&= \braketK{\delta_x}{(\trunc{K^0})^{nm}}{\indicator{B_j}} 
+ \braketK{\delta_x}{(K^0_\perp)^{nm}}{\indicator{B_j}}\;,
\end{align}
where $K^0_\perp = K^0 - \trunc{K^0}$. The second term on the right-hand side 
decreases like the $nm$th power of the spectral gap $\varrho$. As for the first 
term, it can be written 
\begin{align}
\braketK{\delta_x}{(\trunc{K^0})^{nm}}{\indicator{B_j}}
&= \braketK{\delta_x}{\Pi^0(\trunc{K^0})^{nm}}{\indicator{B_j}} \\
&= \sum_{\ell=1}^N \braket{\delta_x}{\psi_\ell} 
\braketK{\mu_\ell}{(\trunc{K^0})^{nm}}{\indicator{B_j}} \\
&= \sum_{\ell=1}^N \psi_\ell(x) \bigprobin{\mu_\ell}{X_{\tau^{+,nm}_{\MN}} 
\in B_j}\;.
\end{align}
If $x\in B_i$, the term $\ell=i$ can be estimated by~\eqref{eq:thm_reduc2}, 
while the other terms are exponentially small by Proposition~\ref{prop:norms_mupsi}.
\end{proof}


\appendix

\section{Other proofs for Section~\ref{sec:setup}} 
\label{sec:proof_setup} 

\subsection{Proof of Proposition~\ref{prop:ldp_committor}}

Since $I$ is continuous at $(x^\star_i,x^\star_i)$ and $(x^\star_j,x^\star_j)$ 
and $I(x^\star_i,x^\star_i)=I(x^\star_j,x^\star_j)=0$, we can find $\delta>0$ 
such that $I(y_1,y_2) \leqs \eta/6$ for all $y_1, y_2\in B_i$, and similarly 
for points $z_1, z_2 \in B_j$. This implies that 
\begin{equation}
 H(i,j) - \frac{\eta}2 \leqs V(y,z) \leqs H(i,j) + \frac{\eta}2
\end{equation} 
holds for all $y\in B_i$ and all $z\in B_j$. 
Consider now the increasing sequence of events 
\begin{equation}
 \Gamma_n = \bigset{\tau^+_{B_j}(x) < \tau^+_{B_i}(x), \tau^+_{B_j}(x) \leqs n} 
 = \bigcup_{m=1}^n \Bigpar{\bigbrak{(B_i\cup B_j)^c}^{m-1} \times B_j \times 
\X_0^{n-m}}\;.
\end{equation} 
Then the LDP for paths $(x,x_1,\dots,x_n)$ yields 
\begin{equation}
 -\inf_{\mathring\Gamma_n} I(x,\cdot) 
 \leqs \liminf_{\sigma\to0} \sigma^2 \log\fP^{x}(\Gamma_n) 
 \leqs \limsup_{\sigma\to0} \sigma^2 \log\fP^{x}(\Gamma_n) 
 \leqs -\inf_{\bar\Gamma_n} I(x,\cdot)\;.
\end{equation} 
Since $\fP^{x}(\Gamma_n)$ is increasing in $n$, to prove the lower bound, it 
suffices to find $n\geqs2$, points $x_1, \dots, x_{n-1} \in (B_i\cup B_j)^c$ 
and 
$z \in \mathring{B_j}$ such that $ I(x,x_1,\dots,x_{n-1},z) \leqs H(i,j) + 
\eta$. To this end, let $y\in B_i$ and $z\in B_j$ be the points minimizing 
$V$. Since $V(y,z) \leqs H(i,j) + \eta/2$, there exist $n$, $x_1, \dots, 
x_{n-1}$ such that $I(y,x_1,\dots,x_{n-1},z) \leqs H(i,j) + 3\eta/4$.  We can 
assume that $x_1, \dots, x_{n-1} \notin B_i\cup B_j$ since otherwise there 
would 
exist a cheaper way to connect these sets. Replacing $y$ by $x$ 
increases $I$ by at most $\eta/6$, yielding the required path. 

To prove the upper bound, we have to show that for any $n$, and any path 
$(x_1,\dots,x_n)\in\bar\Gamma_n$, $I(x,x_1,\dots,x_n) \geqs H(i,j) - 
\eta$. This follows from the fact that $V(x,y) \geqs H(i,j) - \eta$ for all 
$y\in B_j$, since $V(x,y)$ involves the infimum over a larger set. 
\qed

\subsection{Proof of Proposition~\ref{prop:ldp_EMN}}

In the spirit of~\cite[Chapt.~6, Thm.~5.1]{FreidlinWentzell_book}, we first 
construct a path of finite length $n_0$ from $x$ to $\MN$, whose rate function 
$I$ is bounded by $\eta/2$. In the case where the $\omega$-limit set 
$\omega(x)$ 
is one of the stable fixed points $x^\star_i$, there exists $n_0\in\N$ such that 
$\Pi^{n_0}(x) \in B_i$. Setting $x_n=\Pi^n(x)$, we have $I(x,x_1,\dots,x_n) = 
0$. If $\omega(x)$ is an unstable fixed point $y^\star$, we can find $n_1\in\N$ 
such that $\norm{\Pi^{n_1}(x)-y^\star} \leqs \delta$. Since the stable manifolds 
of all unstable fixed points have codimension at least $1$, they cannot contain 
any open subset of $\X$. Thus there exists a point $y_1$ at distance at most 
$\delta$ from $y^\star$ such that $\omega(y_1)$ is a stable fixed point 
$x^\star_i$. Setting $x_n=\Pi^n(x)$ and $y_n=\Pi^{n-1}(y_1)$, we obtain the 
existence of $n_2\in\N$ such that $y_{n_2}\in B_i$ and 
\begin{equation}
 I(x,x_1,\dots,x_{n_1},y_1,\dots,y_{n_2}) = I(x_{n_1},y_1)\;.
\end{equation} 
The continuity of $I$ at $(y^\star,y^\star)$ implies that we can assume 
$I(x_{n_1},y_1)\leqs\eta/4$ by making $\delta$ small enough. 
The large-deviation lower bound implies that if $n_0=n_1+n_2$, then 
\begin{equation}
\label{eq:proof_LDP_MN} 
 \liminf_{\sigma\to 0} \sigma^2 \log\bigprobin{x}{\tau^+_{\MN}\leqs n_0} 
 \geqs -\frac{\eta}{2}\;,
\end{equation} 
so that there exists $\sigma_0>0$ such that 
$\bigprobin{x}{\tau^+_{\MN}\leqs n_0} \geqs \e^{-\eta/\sigma^2}$
holds for all $x\in\X$ and all $\sigma<\sigma_0$. 
To extend this result to an estimate on expected return times, we use the fact 
that for any sets $A, B, C\in\cS_0$ such that $B\cap C=\emptyset$, one has 
\begin{equation}
\label{eq:three_set_expectation} 
 \bigexpecin{A}{\tau^+_B} \leqs 
 \bigexpecin{A}{\tau^+_{B\cup C}} + 
 \bigprobin{A}{\tau^+_C < \tau^+_B} \bigexpecin{C}{\tau^+_B}\;,
\end{equation} 
where we write $\probin{A}{\cdot} = \sup_{x\in A}\probin{x}{\cdot}$. 
For a proof, see for instance~\cite[Lem.~8.9]{BaudelBerglund_2017} (the 
proof only requires $B\cap C=\emptyset$). 
Taking $A=\X$, $B=\MN$ and $C=\X\setminus\MN$, 
bounding $\expecin{C}{\tau^+_B}$ by $\expecin{A}{\tau^+_B}$,
and rearranging, we obtain 
\begin{equation}
 \bigexpecin{\X}{\tau^+_{\MN}} 
 \leqs \frac{\bigexpecin{\X}{\tau^+_{\X}}}
 {\bigprobin{\X}{\tau^+_{\MN} < \tau^+_{\X\setminus\MN}}}\;.
\end{equation} 
The same relation holds when the $\tau^+$ are replaced by the return times 
$\hat\tau^+$ of the diluted process $(X_{nn_0})_{n\in\N}$. This yields 
\begin{equation}
\label{eq:proof_EXMN_n0} 
 \bigexpecin{\X}{\tau^+_{\MN}} 
 \leqs n_0 \bigexpecin{\X}{\hat\tau^+_{\MN}}
 \leqs \frac{n_0\bigexpecin{\X}{\hat\tau^+_{\X}}}
 {\bigprobin{\X}{\hat\tau^+_{\MN} < \hat\tau^+_{\X\setminus\MN}}}\;.
\end{equation} 
Observe that for any $x\in\X$, 
\begin{align}
\bigprobin{x}{\hat\tau^+_{\MN} < \hat\tau^+_{\X\setminus\MN}}
&{}\geqs{} \bigprobin{x}{1 = \hat\tau^+_{\MN} < \hat\tau^+_{\X\setminus\MN}} \\
&{}={} \bigprobin{x}{X_{n_0}\in\MN} \\
&{}={} \bigprobin{x}{\tau^+_{\MN} \leqs n_0} 
- \bigprobin{x}{\tau^+_{\MN}\leqs n_0, X_{n_0}\notin\MN} \\
&{}\geqs{} \bigprobin{x}{\tau^+_{\MN} \leqs n_0} 
\Bigbrak{1 - \sup_{k\leqs n_0} \bigprobin{\MN}{X_k\notin\MN}}\;.
\end{align}
The supremum is exponentially small by Lemma~\ref{lem:exit_M}. 
Since $n_0$ is independent of $\sigma$ and 
$\bigexpecin{\X}{\tau^+_{\X}}$ is uniformly bounded, the result follows.
\qed


\section{Proofs for Section~\ref{sec:applic}} 
\label{sec:proof_applic} 


\subsection{Proof of Proposition~\ref{prop:pos_gaussian_map}} 
\label{ssec:proof_pos_gaussian_map} 

An important tool in the proof is the following coupling argument. 

\begin{proposition}[Coupling argument]
\label{prop:coupling} 
Let $K_A$ be a submarkovian kernel on a set $A$, and denote its 
killing time by $\tau_{A^c}$. Assume that there exist 
constants $r, \eta>0$ such that the density $k_A$ of $K_A$ satisfies the 
Harnack inequality 
\begin{equation}
 \sup_{x\in A\colon \norm{x-x_0}\leqs r} k_A(x,y) 
 \leqs (1+\eta) \inf_{x\in A\colon \norm{x-x_0}\leqs r} k_A(x,y)
\end{equation} 
for all $x_0, y\in A$. Let $(\hat X^{x}_n)_{n\geqs0}$ be the process 
with kernel $K_A$ conditioned on staying in $A$, defined by 
\begin{equation}
 \bigprob{\hat X^{x}_n\in B} = 
\frac{K^n_A(x,B)}{K^n_A(x,A)} 
\end{equation}
for any Borel set $B\subset A$. Assume that for any $x_1\neq x_2\in A$, there 
exists a coupling between $(\hat X^{x_1}_n)_{n\geqs0}$ and $(\hat 
X^{x_2}_n)_{n\geqs0}$ such that the stopping time 
\begin{equation}
 N(x_1, x_2) = \inf\bigsetsuch{n\geqs1}{\norm{\hat X^{x_2}_n - \hat X^{x_1}_n} 
\leqs r}
\end{equation} 
is almost surely finite, and define 
\begin{equation}
 \rho_n = \sup_{x_1\neq x_2\in A} \bigprob{N(x_1,x_2) > n}\;.
\end{equation} 
Then $k_A$ satisfies for every $n\in\N$ a uniform positivity condition with 
parameters $n$ and 
\begin{equation}
\label{eq:L} 
 L = \dfrac{(1+\eta)^2 + \rho_{n-1}\displaystyle\sup_{y\in A} 
 \Biggpar{\dfrac{\sup_{x\in A}k_A(x,y)}{\inf_{x\in A}k_A(x,y)}}}
{\displaystyle\inf_{x\in A}\bigprobin{x}{\tau_{A^c} > n}}\;.
\end{equation} 
\end{proposition}
\begin{proof}
See~\cite[Prop.~5.9]{berglund2014noise} 
and~\cite[Prop.~5.4]{BaudelBerglund_2017}.
\end{proof}

Fix $1\leqs i\leqs k\leqs N$. 
To apply the above coupling argument, we need to 
estimate the various terms appearing in~\eqref{eq:L}. We first claim that for 
any $\eta>0$, there exist $C, r>0$ such that the two relations 
\begin{align}
\label{eq:proof_pos_01} 
\sup_{x\in B_i}\ptrace{k_{\sigma,B_i}}{\cM}(x,y) 
&\leqs \e^{C/\sigma^2} 
\inf_{x\in B_i}\ptrace{k_{\sigma,B_i}}{\cM}(x,y) \\
\sup_{x\in B_i \colon \norm{x-x_0} \leqs r\sigma^2}
\ptrace{k_{\sigma,B_i}}{\cM}(x,y) 
&\leqs (1 + \eta) 
\inf_{x\in B_i \colon \norm{x-x_0} \leqs r\sigma^2}
\ptrace{k_{\sigma,B_i}}{\cM}(x,y)
\label{eq:proof_pos_02} 
\end{align}
hold for all $x_0, y\in B_i$. 
First note that the Gaussian density~\eqref{eq:gaussian_density} of the 
original kernel satisfies 
\begin{equation}
 \frac{k_\sigma(x_1,y)}{k_\sigma(x_2,y)}
 = \exp\Biggset{\frac{I(x_2,y)-I(x_1,y)}{\sigma^2}}\;,
\end{equation} 
where 
$I(x_2,y)-I(x_1,y) = \pscal{\Sigma^{-1}y}{\Pi(x_1)-\Pi(x_2)}
 + \frac12\pscal{\Pi(x_2)}{\Sigma^{-1}\Pi(x_2)}
 - \frac12\pscal{\Pi(x_1)}{\Sigma^{-1}\Pi(x_1)}$.
This quantity is bounded by a constant $C$ for all $x_1, x_2\in B_i$ and 
$y$ in a bounded set, and has order $r\sigma^2$ if in addition $\norm{x_1-x_2} 
\leqs r\sigma^2$. Hence $k_\sigma$ satisfies~\eqref{eq:proof_pos_01} 
and~\eqref{eq:proof_pos_02} if $r=r(\eta)$ is small enough. 

In order to extend this to $\ptrace{k_\sigma}{\cM}$, we use the fact that 
for all $n\in\N$ and $x_1, y\in B_i$, we have 
\begin{equation}
 \bigprobin{x_1}{\tau^+_{B_i} = n} k^n_\sigma(x_1,y) 
 = \int_{B_i^c} k_\sigma(x_1,z) 
 \bigprobin{z}{\tau^+_{B_i} = n-1} k^{n-1}_\sigma(z,y) \6z\;.
\end{equation} 
The Laplace method shows that the integral is dominated by $z$ of order $1$ at 
most. We can thus bound $k_\sigma(x_1,z)$ above 
by $\e^{C/\sigma^2}k_\sigma(x_2,z)$ for any $x_2\in B_i$, and by 
$(1+\eta)k_\sigma(x_2,z)$ if $\norm{x_1-x_2} \leqs r\sigma^2$. Together with 
the expression~\eqref{eq:density_trace} for the density of the trace process, 
this shows that $\ptrace{k_\sigma}{\cM}$ also 
satisfies~\eqref{eq:proof_pos_01} and~\eqref{eq:proof_pos_02}. 

Regarding the effect of the killing, denote by 
$\smash{\ptrace{\tau^+_{B_i^c}}{\cM}}$ the killing time of the trace process 
and observe that Proposition~\ref{prop:ldp_committor} yields 
\begin{equation}
 \bigprobin{B_i}{\ptrace{\tau^+_{B_i^c}}{\cM} = 1} 
 = \bigprobin{B_i}{\tau^+_{\cM\setminus B_i} < \tau^+_{B_i}} 
 \leqs \e^{-H/\sigma^2}
\end{equation}
for an $H>0$. By the Markov property, we get 
$\smash{\bigprobin{B_i}{\ptrace{\tau^+_{B_i^c}}{\cM} \leqs n}} \leqs n 
\e^{-H/\sigma^2}$ for any $n\in\N$. Since $k_\sigma(x,y)$ is bounded below by 
$\e^{-c\delta_0^2/\sigma^2}$ for $x, y\in B_i$, this shows that the killing 
has a negligible effect for $\delta_0$ small enough. It also shows that the 
denominator in~\eqref{eq:L} is close to $1$. 

It thus remains to show that $\rho_{n-1}$ in~\eqref{eq:L} can be made 
exponentially small for an $n$ of order $\log(\sigma^{-1})$. 
Let $(X^{x_1}_n)_{n\geqs0}$ and $(X^{x_2}_n)_{n\geqs0}$ denote the original 
processes driven by the same noise $(\xi_n)_{n\geqs1}$, and starting 
respectively from $x_1$ and $x_2$. With this coupling, writing $Y_n = X^{x_2}_n 
- X^{x_1}_n$, we see that 
\begin{equation}
 Y_{n+1} = \Pi(X^{x_1}_n + Y_n) - \Pi(X^{x_1}_n) 
 = A_n Y_n + b_n(Y_n)\;,
\end{equation} 
where $A_n = \partial_x\Pi(X^{x_1}_n)$ and $\norm{b_n(y)} \leqs M \norm{y}^2$ 
for bounded $y$ and some constant $M>0$. Since $\partial_x\Pi(x^\star_i)$ has 
spectral radius strictly smaller than $1$, there exists $0 < \varrho_1 < 1$ such 
that $A_n$ has spectral radius bounded by $\varrho_1$ for any $n$ such that 
$X^{x_1}_n \in B_i$. Thus there exists a norm $\norm{\cdot}'$, equivalent to the 
Euclidean norm, such that $\norm{A_n y}' \leqs \varrho_2 \norm{y}'$ for these 
$n$, where $\varrho_2 < 1$. Taking $\delta_0$ small enough that 
$M\norm{x_2-x_1}' < 1-\varrho_2$, we conclude that $\norm{Y_1}' \leqs \varrho 
\norm{x_2-x_1}'$. 

Since $\Pi(B_i) \subset B_i$ (where the inclusion is strict for $\delta_0$ 
small enough), there exists $\kappa>0$ such that $\probin{B_i}{X_1 \notin 
B_i} \leqs \e^{-\kappa/\sigma^2}$, and thus $\probin{B_i}{\exists \ell \leqs 
n \colon X_\ell \notin B_i} \leqs n\e^{-\kappa/\sigma^2}$ for any $n\in\N$. 
Hence the coupled trace processes conditioned on staying in $B_i$ satisfy 
\begin{equation}
\label{eq:proof_pos_03} 
  \bigprob{\norm{\hat X^{x_2}_n - \hat X^{x_1}_n}' > \varrho^n\norm{x_2 - x_1}'}
 \leqs \frac{2n\e^{-\kappa/\sigma^2}}{1 - 2n\e^{-\kappa/\sigma^2}}
 \leqs 3n\e^{-\kappa/\sigma^2}
\end{equation} 
for $\sigma$ small enough. Let $N(x_1,x_2) = \inf\setsuch{n\geqs1}{\norm{\hat 
X^{x_2}_n - \hat X^{x_1}_n}' \leqs r(\eta)\sigma^2}$, and let $n_1(\sigma)$ be 
such that $\diam(B_i)\varrho^{n_1(\sigma)} \leqs r(\eta)\sigma^2$. Note that 
$n_1(\sigma)$ has order $\log(\sigma^{-1})$, and that $\prob{N(x_1,x_2) > 
n_1(\sigma)}$ is bounded above in~\eqref{eq:proof_pos_03}. Applying the Markov 
property at times which are multiples of $n_1(\sigma)$, we obtain 
\begin{equation}
 \rho_{\ell n_1(\sigma)}
 = \bigprob{N(x_1,x_2) > \ell n_1(\sigma)} \leqs 
\bigpar{3n_1(\sigma)\e^{-\kappa/\sigma^2}}^\ell\;.
\end{equation} 
Choosing $\ell$ such that $\ell\kappa > C$, the result follows with 
$n_0(\sigma) = \ell n_1(\sigma)$, taking $\eta$ small enough.  
\qed


\subsection{Proof of Proposition~\ref{prop:EMN_gaussian}} 
\label{ssec:proof_EMN_gaussian} 

The proof is essentially an adaptation to the discrete-time setting of results 
in~\cite[Sect.~8.2 and~8.3]{BaudelBerglund_2017}, which rely in part on 
methods from~\cite[Chapt.~5]{Berglund_Gentz_book}. Since several proofs 
simplify in the present setting, we believe it is worth giving details 
here. 
In view of the upper bound~\eqref{eq:proof_EXMN_n0} for 
$\smash{\bigexpecin{\X}{\tau^+_{\MN}}}$ and the fact that 
$\bigexpecin{\X}{\tau^+_\X}$ is bounded by\Cref{ass:REC}, it is sufficient to 
show that $\smash{\bigprobin{\X}{\tau^+_{\MN} \leqs n_0}}$ is bounded away 
from $1$ for some $n_0$ of order $\log(\sigma^{-1})$.

Given $x\in\X$, we first give an estimate for the probabillity of the sample 
path $(X_n)_{n\geqs0}$ starting in $x$ deviating from the deterministic orbit 
$(X^{\det}_n)_{n\geqs0}$ defined by $X^{\det}_n = \Pi^n(x)$. 

\begin{lemma}
\label{lem:stable} 
There exist constants $\kappa>0$ and $h_0>0$ such 
that for any $n\in\N$, 
\begin{equation}
 \biggprobin{x}{\max_{1\leqs i\leqs n} \norm{X_i - X^{\det}_i} > h}
 \leqs n \exp\Biggset{-\frac{\kappa h^2}{\sigma^2G^{(2)}_n}}
\end{equation} 
holds whenever $0 < h < h_0/G^{(1)}_n$, where 
$ G^{(\ell)}_n = \sum_{i=0}^{n-\ell} D_n^{\ell i}$ with  
$ D_n = \max_{1\leqs i\leqs n} \norm{\partial_x \Pi(X^{\det}_i)}$.
\end{lemma}
\begin{proof}
The difference $Y_n = X_n - X^{\det}_n$ satisfies 
$Y_0=0$ and 
\begin{align}
 Y_{n+1} &= 
 \Pi(X^{\det}_n + Y_n) - \Pi(X^{\det}_n) + \sigma\xi_{n+1} \\
 &= A_n Y_n + b(Y_n) + \sigma\xi_{n+1}\;,
\end{align} 
where we have set $A_n = \partial_x \Pi(X^{\det}_n)$, and $\norm{b(y)} \leqs 
M\norm{y}^2$ for bounded $y$ and some $M>0$. 
Consider first the linearized equation 
\begin{equation}
 Y^0_{n+1} = A_n Y^0_n + \sigma\xi_{n+1}\;,
 \qquad Y^0_0 = 0\;.
\end{equation} 
Its solution can be written 
\begin{equation}
 Y^0_n = \sigma\sum_{i=1}^n B_{ni}\xi_i\;,
\end{equation} 
where $B_{ni} = A_{n-1}\dots A_i$ if $i<n$ and $B_{nn}=\id$ is the identity 
matrix. Thus $Y^0_n$ is a centred Gaussian random variable with covariance 
matrix $\sigma^2 \Sigma_n$, where 
\begin{equation}
 \Sigma_n = \sum_{i=1}^n B_{ni}\Sigma \transpose{B_{ni}}\;.
\end{equation} 
We have $\norm{B_{ni}} \leqs D_n^{n-i}$ and $\norm{\Sigma_n} \leqs \norm{\Sigma}
G^{(2)}_n$. This implies that $\prob{\norm{Y^0_n}> h} \leqs \e^{-\kappa_0 
h^2/(G^{(2)}_n\sigma^2)}$ for some $\kappa_0>0$, and thus 
\begin{equation}
\label{eq:proof_Y0} 
 \biggprob{\max_{1\leqs i\leqs n} \norm{Y^0_i} > h}
 = \fP \Biggpar{\bigcup_{i=1}^n \bigset{\norm{Y^0_i} > h}}
 \leqs n \exp\Biggset{-\frac{\kappa_0 h^2}{G^{(2)}_n \sigma^2}}\;.
\end{equation} 
To extend this estimate to $Y_n$, we write $Y_n = Y^0_n + R_n$ and note that we 
have 
\begin{equation}
 R_{n+1} = A_n R_n + b(Y_n) 
 \qquad \Rightarrow \qquad 
 R_n = \sum_{i=2}^n B_{ni} b(Y_{i-1})\;.
\end{equation} 
Setting $\tau = \inf\setsuch{n\geqs1}{\norm{Y_n}>h}$, we have 
$\norm{b(Y_{n\wedge\tau})} \leqs Mh^2$ and 
$\norm{R_{n\wedge\tau}}\leqs G^{(1)}_n Mh^2$. 
For any decomposition $h = H_0 + H_1$ with $H_0, H_1>0$, we have 
\begin{equation}
 \bigprob{\tau < n} 
 \leqs \biggprob{\max_{1\leqs i\leqs n} \norm{Y^0_i} > H_0} 
 + \biggprob{\max_{1\leqs i\leqs n\wedge\tau} \norm{R_i} > H_1}\;.
\end{equation} 
The first term on the right-hand side can be estimated 
with~\eqref{eq:proof_Y0}, while the second one vanishes if $H_1 \geqs G^{(1)}_n 
Mh^2$. The result thus follows by setting $H_0 = h(1-G^{(1)}_nMh)$. 
\end{proof}

\begin{corollary}
\label{cor:stable} 
Let $x\in\X$ be such that $\omega(x)$ is a stable fixed point 
$\smash{x^\star_i}$.
Then there exist $n(x) < \infty$ and $\kappa(x)>0$ such that\/
$\bigprobin{x}{\tau^+_{\MN} \geqs n(x)} \leqs n(x)\e^{-\kappa(x)/\sigma^2}$.
\end{corollary}
\begin{proof}
By definition of $\omega$-limit sets, there exists $n(x)$ such that 
$\norm{\Pi^{n(x)}(x) - x^\star_i} < \delta/2$. It is thus sufficient to apply  
Proposition~\ref{lem:stable} with $h=\delta/2$. 
\end{proof}

This bound deteriorates when $x$ approaches an unstable fixed point of $\Pi$ 
(or the stable manifold of such a fixed point), because $n(x)$ diverges and 
$\kappa(x)$ tends to $0$. We thus have to treat these cases separately. In 
doing so, we will repeatedly use the following elementary estimate.

\begin{lemma}
\label{lem:three_set_prob} 
Let $A,B$ be two disjoint sets in $\cS_0$. 
Then for any $n_1, n_2\in\N$, 
\begin{equation}
 \bigprobin{A\cup B}{\tau^+_{(A\cup B)^c} \geqs n_1+n_2} 
 \leqs \bigprobin{A}{\tau^+_{A^c} \geqs n_1} 
 + \bigprobin{B}{\tau^+_{B^c} \geqs n_2}
 + \bigprobin{B}{\tau^+_A < \tau^+_{(A\cup B)^c}}\;.
\end{equation} 
\end{lemma}
\begin{proof}
When starting in $B$, consider separately the cases $\tau^+_{(A\cup B)^c} = 
\tau^+_{A^c}$ and $\tau^+_{(A\cup B)^c} = \tau^+_{B^c}$. 
When starting in $A$, distinguish the cases $\tau^+_{A^c} \geqs n_1$ 
and $\tau^+_{A^c} < n_1$, and use the bound for starting points in $B$.
\end{proof}

To estimate exit probabilities from the neighbourhood of 
an unstable equilibrium point, we proceed in two steps, considering first the 
exit from a small neighbourhood of size $\sigma^{3/4}$, and then the exit from 
a larger neighbourhood of size $\delta$. 

\begin{lemma}
\label{lem:exit_S} 
Let $\cS$ be a neighbourhood of diameter $\sigma^{3/4}$ of an unstable fixed 
point $z^\star_j$. Then there exist constants $c_1, C_1>0$ such that 
\begin{equation}
 \bigprobin{\cS}{\tau^+_{\cS^c} > c_1 \log(\sigma^{-1})} 
 \leqs C_1 \log(\sigma^{-1})\sigma^{1/2}\;.
\end{equation} 
\end{lemma}
\begin{proof}
We may assume that $z^\star_j = 0$. 
Let $\lambda_+$ be the module of the largest eigenvalue of 
$\partial_x\Pi(z^\star)$, and let $m$ be its multiplicity. There exists a linear 
change of variables $X_n \mapsto (Y_n,Z_n)$ such that 
\begin{align}
Y_{n+1} &= A_+ Y_n + b_+(Y_n,Z_n) + \sigma\xi^+_{n+1}\;, 
& Y_0 &= y\;, \\
Z_{n+1} &= A_- Z_n + b_-(Y_n,Z_n) + \sigma\xi^-_{n+1}\;,
& Z_0 &= z\;, 
\label{eq:proof_stable_unstable} 
\end{align}
where $A_+$ is a square matrix of size $m$, all of whose eigenvalues are equal 
to $\lambda_+$, all eigenvalues of $A_-$ are strictly smaller in module than 
$\lambda_+$, $\norm{b_\pm(y,z)} \leqs M(\norm{y^2}+\norm{z}^2)$ for 
bounded $y$ and $z$, and the $\xi^\pm_n$ are nondegenerate Gaussian random 
variables. Let $Y^0_n$ obey the linearized dynamics 
$Y^0_{n+1} = A_+ Y^0_n + \sigma\xi^+_{n+1}$. Similarly to the 
Lemma~\ref{lem:stable}, we have $Y_n = Y^0_n + R_n$, where 
\begin{equation}
 Y^0_n = A_+^n y + \sigma\sum_{i=1}^n A_+^{n-i} \xi^+_i\;, 
 \qquad 
 R_n = \sum_{i=2}^n A_+^{n-i} b_+(Y_{i-1},Z_{i-1})\;.
\end{equation} 
A similar decomposition $Z_n = Z^0_n + Q_n$ holds for the second component. 
Let $\Sigma_+$ denote the covariance matrix of the $\xi^+_i$.
The law of $Y^0_n$ is Gaussian with covariance matrix 
\begin{equation}
 \cov(Y^0_n) = \sigma^2 \sum_{i=1}^n A_+^{n-i} \Sigma_+ (\transpose{A_+})^{n-i}
\end{equation} 
We have $\det\cov(Y^0_n) \geqs c(\sigma\lambda_+^n)^{2m}$ for some $c>0$, so 
that there exists $C>0$ such that 
\begin{equation}
 \bigprob{\norm{Y^0_n} < h} \leqs C\Biggpar{\frac{h}{\sigma\lambda_+^n}}^m  
\end{equation} 
for all $n\in\N$ and $h>0$. Setting $X^0_i=(Y^0_i,Z^0_i)$ we have for any 
$h_1>0$ 
\begin{align}
\biggprob{\max_{1\leqs i\leqs n} \norm{X_i} < h}
& \leqs \biggprob{\max_{1\leqs i\leqs n} \norm{X^0_i} < h+h_1}
+ \biggprob{\max_{1\leqs i\leqs n} \norm{(R_i,Q_i)} \geqs h_1,
\max_{1\leqs i\leqs n} \norm{X_i} < h}\;, \\
& \leqs \biggprob{\norm{Y^0_n} < h+h_1}
+ \biggprob{\max_{1\leqs i\leqs n} \norm{(R_i,Q_i)} \geqs h_1,
\max_{1\leqs i\leqs n} \norm{X_i} < h}\;.
\end{align}
The second term on the right-hand side vanishes if we set $h_1 = 
C_1Mh^2\lambda_+^n$ for a sufficiently large constant $C_1$. Choosing 
$h=\sigma^{3/4}$ and $n$ such that $\lambda_+^n \geqs \sigma^{-3/4}$, we obtain 
the result. 
\end{proof}

\begin{lemma}
\label{lem:exit_U} 
Let $\cU$ be a neighbourhood of diameter $\delta$ of an unstable fixed point 
$z^\star_j$. Then there exist constants $c_2, C_2>0$ such that 
\begin{equation}
 \bigprobin{\cU}{\tau^+_{\cU^c} > c_2 \log(\sigma^{-1})}
 \leqs C_2 \log(\sigma^{-1})\sigma^{1/2}\;.
\end{equation} 
\end{lemma}
\begin{proof}
We may use a similar coordinate system as in~\eqref{eq:proof_stable_unstable}, 
except that now $Y$ contains all unstable directions, while $Z$ contains the 
marginally stable and stable ones. The center-stable manifold theorem allows us 
to assume that $b_+(0,z) = 0$ and $\norm{b_+(y,z)} \leqs M(\norm{y}^2 + 
\norm{y}\norm{z})$ in $\cU$ for some $M>0$. The Lyapunov function $U_n = 
\norm{Y_n}^2$ satisfies 
\begin{equation}
 U_{n+1} = \norm{A_+ Y_n}^2 
 + \bigbrak{2\pscal{b_+}{A_+ Y_n} + \norm{b_+}^2}
 + 2\sigma \pscal{A_+ Y_n + b_+}{\xi^+_{n+1}}
 + \sigma^2 \norm{\xi^+_{n+1}}^2\;.
\end{equation} 
All eigenvalues of $A_+$ have a module strictly larger that $1$, 
showing that $\norm{A_+ Y_n}^2 \geqs \lambda_+ U_n$ for some $\lambda_+ > 1$. 
The term in square brackets has order $U_n^{3/2} + U_n\norm{Z_n}$. Thus for 
small enough $\delta$, there exists $\bar\lambda_+ > 1$ such that 
\begin{equation}
 U_{n+1} \geqs \bar\lambda_+ U_n + \sigma g(X_n) \eta_{n+1}
 + \sigma^2 \norm{\xi^+_{n+1}}^2\;,
\end{equation} 
where $\norm{g(x)} \leqs M_1 U_n^{1/2}$ for some $M_1>0$, and $\eta_{n+1}$ is a 
centred Gaussian random variable of bounded variance. 
Let $\cK = \setsuch{(y,z)\in\cU}{\sigma^{3/4} \leqs \norm{y} \leqs \delta}$. 
For $n \leqs \tau^+_{\cK^c}$, we obtain that $V_n = U_n^{1/2}$ satisfies 
\begin{equation}
 V_{n+1} \geqs \bar\lambda_+^{1/2} V_n + \sigma\bar g(X_n) \eta_{n+1}\;,
\end{equation} 
where $\bar g$ is bounded in $\cK$. It follows that for $n\leqs\tau^+_{\cU^c}$, 
\begin{equation}
 V_n \geqs \bar\lambda_+^{n/2} \bigbrak{V_0 + \sigma\zeta_n}\;,
\end{equation} 
where $V_0 \geqs \sigma^{3/4}$ and $\zeta_n = \sum_{i=1}^n 
\bar\lambda_+^{(n-i)/2} \bar g(X_{i-1}) \eta_i$ has bounded variance. 
Chebyshev's inequality shows that 
\begin{equation}
\biggprob{\min_{1\leqs i\leqs n\wedge\tau^+_{\cK^c}}
\frac{\zeta_i}{\bar\lambda_+^{i/2}-1} < -\sigma^{3/4} } 
\leqs nC \sigma^{1/2}
\end{equation} 
for some $C>0$. Taking $n$ of order $\log(\sigma^{-1})$ such that 
$\bar\lambda_+^{i/2} > \delta/\sigma^{3/4}$, this yields the existence of 
constants $c_1, C_1>0$ such that 
\begin{align}
 \bigprobin{\cK}{\tau^+_{\cK^c} > c_1\log(\sigma^{-1})} 
 &\leqs C_1 \log(\sigma^{-1})\sigma^{1/2}\;, \\
 \bigprobin{\cK}{\tau^+_\cS < \tau^+_{\cU^c}} 
 &\leqs C_1 \log(\sigma^{-1})\sigma^{1/2}\;. 
 \label{eq:proof_bound_committor} 
\end{align} 
Applying Proposition~\ref{lem:three_set_prob} with $A=\cS$ and $B=\cK$ yields the claimed 
result with $\cK\cup\cS$ instead of $\cU$. The result can be extended to $\cU$ 
by showing that Proposition~\ref{lem:exit_S} also applies to the larger set 
$\cU\setminus\cK=\setsuch{(y,z)\in\cU}{\norm{y}\leqs\sigma^{3/4}}$, using the 
better bounds on $b_+$ due to the centre-stable manifold theorem, and analysing 
a slightly more general recursion for $Y_n$ with time-dependent linear part.
\end{proof}

To finish the proof, we have to deal with the possible existence of heteroclinic 
orbits. Denote the unstable fixed points by ${z^\star_1, \dots, 
z^\star_M}$. Let $\cU_i$ be a ball of diameter $\delta$ centred in $z^\star_i$, 
with $\delta$ small enough for Proposition~\ref{lem:exit_U} to apply. We denote the union 
of all $\cU_i$ by $\cU$. Define 
\begin{align}
\tau^{\det}_A(x) &= \inf\bigsetsuch{n\geqs1}{\Pi^n(x)\in A}\;, \\
 \cA_i &= \bigsetsuch{x\in\cX\setminus\cU}{X_{\tau^{\det}_{\cU}} \in \cU_i}\;.
\end{align} 
The set $\cA_i$ contains part of the stable manifold of $z^\star_i$. Note that 
$\cA_i$ contains no fixed points of $\Pi$, showing, by Proposition~\ref{lem:stable}, that 
$\smash{\bigprobin{\cA_i}{\tau^+_{\cA_i^c} \geqs n}}$ is exponentially small 
for some bounded $n$. Furthermore, the proof of 
Proposition~\ref{lem:exit_U}, in particular~\eqref{eq:proof_bound_committor}, shows that 
\begin{equation}
 \bigprobin{\cU_i}{\tau^+_{\cA_i} < \tau^+_{(\cA_i\cup\cU_i)^c}} 
\leqs C_1 \log(\sigma^{-1})\sigma^{1/2}\;.
\end{equation} 
This is because deterministic orbits starting on the boundary of $\cK$ on 
which $Y=\delta$ cannot enter $\cA_i$ (recall that there are no heteroclinic 
cycles). Thus Proposition~\ref{lem:three_set_prob} shows that there exist
constants $c_3,C_3>0$ such that 
\begin{equation}
 \bigprobin{\cU_i\cup\cA_i}{\tau^+_{(\cU_i\cup\cA_i)^c} > c_3 
\log(\sigma^{-1})} \leqs C_3\log(\sigma^{-1})\sigma^{1/2}\;.
\end{equation} 
When leaving $\cU_i$, it may happen that a trajectory enters the domain of 
attraction $\cA_j$ of another unstable fixed point, due to the existence of a 
heteroclinic orbit from $z^\star_i$ to $z^\star_j$. In this case we write 
$i\prec j$. Extending this relation by transitivity yields a strict partial 
order relation, owing to the fact that there are no heteroclinic cycles. 
Proposition~\ref{lem:stable} implies that whenever $i\prec j$ or $i$ and $j$ are not 
related, the probability, when starting from $\cU_j\cup\cA_j$, to hit 
$\cU_i\cup\cA_i$ before $(\cU_j\cup\cA_j)^c$ is exponentially small. 
Repeated application of Proposition~\ref{lem:three_set_prob} shows that, if 
$\hat\cU = \bigcup_{i=1}^M (\cU_i\cup\cA_i)$, then 
\begin{equation}
 \bigprobin{\hat\cU}{\tau^+_{\hat\cU^c} > c_4 
\log(\sigma^{-1})} \leqs C_4\log(\sigma^{-1})\sigma^{1/2}
\end{equation} 
holds for constants $c_4, C_4 > 0$. 
In $\X\setminus\hat\cU$, we can apply Corollary~\ref{cor:stable} with a uniformly 
bounded $n(x)$, which finishes the proof, applying one last time 
Proposition~\ref{lem:three_set_prob}.
\qed


\bibliographystyle{amsalpha}
{\small \bibliography{ref}}

\newpage

{\small \tableofcontents}

\vfill

\bigskip\bigskip\noindent
{\small
Institut Denis Poisson (IDP) \\ 
Universit\'e d'Orl\'eans, Universit\'e de Tours, CNRS -- UMR 7013 \\
B\^atiment de Math\'ematiques, B.P. 6759\\
45067~Orl\'eans Cedex 2, France \\
{\it E-mail address: }
{\tt nils.berglund@univ-orleans.fr}

\end{document}